\documentclass[11pt]{amsart}
\pdfoutput=1
\usepackage[leqno]{amsmath}
\usepackage{epsfig}
\usepackage{epsfig}
\usepackage{color}
\usepackage{amsmath}
\usepackage{amssymb}
\newtheorem{lemma}{Lemma}
\newtheorem{proposition}{Proposition}
\newtheorem{theorem}{Theorem}
\newtheorem{corollary}{Corollary}

\theoremstyle{definition}

\usepackage{color}
\newcommand{\commentout}[1]{}
\usepackage{graphicx}

\begin{document}

\thispagestyle{empty}

\centerline{\Large\bf Combinatorics and geometry of finite and infinite squaregraphs}

\vspace{6mm}

\centerline{{\sc Hans--J\"urgen Bandelt$^{\small 1}$,}  {\sc Victor
Chepoi$^{\small 2}$,} and {\sc David Eppstein}$^{\small 3}$}

\vspace{3mm}

\centerline{$^{1}$Dept. of Mathematics, University of Hamburg}
\centerline{Bundesstr. 55, D-20146 Hamburg, Germany}
\centerline{bandelt@math.uni-hamburg.de}

\medskip
\centerline{$^{2}$Laboratoire d'Informatique Fondamentale,}
\centerline{Universit\'e d'Aix-Marseille,}
\centerline{Facult\'e des Sciences de Luminy,} \centerline{F-13288
Marseille Cedex 9, France} \centerline{chepoi@lif.univ-mrs.fr}

\medskip
\centerline{$^3$ Computer Science Department,}
\centerline{University of California, Irvine} \centerline{Irvine  CA
92697-3435, USA} \centerline{eppstein@ics.uci.edu}

\vspace{15mm}
\begin{footnotesize} \noindent {\bf Abstract.} Squaregraphs were originally defined as finite plane graphs in which all inner faces are quadrilaterals (i.e.,
4-cycles) and all inner vertices (i.e., the vertices not incident
with the outer face) have degrees larger than three. The planar dual
of a finite squaregraph is determined by a triangle-free chord diagram of
the unit disk, which could alternatively be viewed as a
triangle-free line arrangement in the hyperbolic plane. This representation carries over to infinite plane graphs with finite vertex degrees
in which the balls are finite squaregraphs.
Algebraically, finite squaregraphs are median graphs for which the duals
are finite circular split systems. Hence squaregraphs are at the crosspoint
of two dualities, an algebraic and a geometric one,
and thus lend themselves to several combinatorial interpretations
and structural characterizations. With these and the 5-colorability
theorem for circle graphs at hand, we prove that
every squaregraph can be isometrically embedded into the Cartesian
product of five trees. This embedding result can also be extended to the infinite case without reference to an
embedding in the plane and without any cardinality restriction when formulated for median graphs free of cubes
and further finite obstructions.  Further, we exhibit a class
of squaregraphs that can be embedded into the product of three trees
and we characterize those squaregraphs that are embeddable into the product of just two trees. Finally, finite squaregraphs
enjoy a number of algorithmic features that do not extend to arbitrary median graphs. For instance, we show that
median-generating sets of finite squaregraphs can be computed in polynomial time, whereas, not unexpectedly, the
corresponding problem for median graphs turns out to be NP-hard.
\end{footnotesize}

\section{Avant-propos}
\label{sec:intro}

A \emph{finite squaregraph} is a finite plane graph $G=(V,E)$ (i.e., a finite graph drawn in the plane with no edges crossing) in which all inner faces are quadrilaterals and in which
all inner vertices have four or more incident edges. Three examples are shown in Figure~\ref{fig:examples}. In most
applications one is only interested in the squaregraphs that are
2-connected, and 2-connectivity is sometimes tacitly assumed in those contexts.
The smallest 2-connected squaregraph is then the 4-cycle or
quadrilateral. The blocks of an arbitrary squaregraph $G$ are either
2-connected or bridges (i.e., copies of the two-vertex complete graph
$K_2$).

\begin{figure}[t]
\begin{tabular}{ccc}
\raisebox{0.5in}{\includegraphics[scale=0.2]{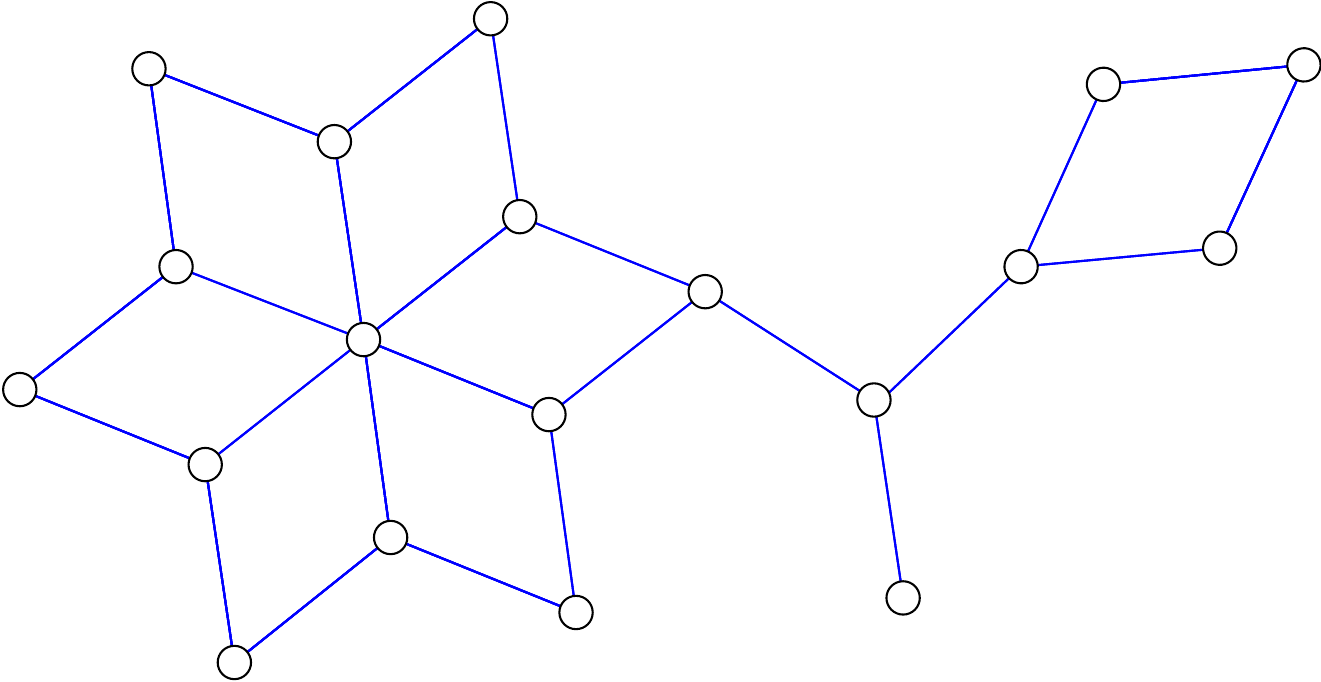}}&
\multicolumn{2}{c}{\includegraphics[scale=0.4]{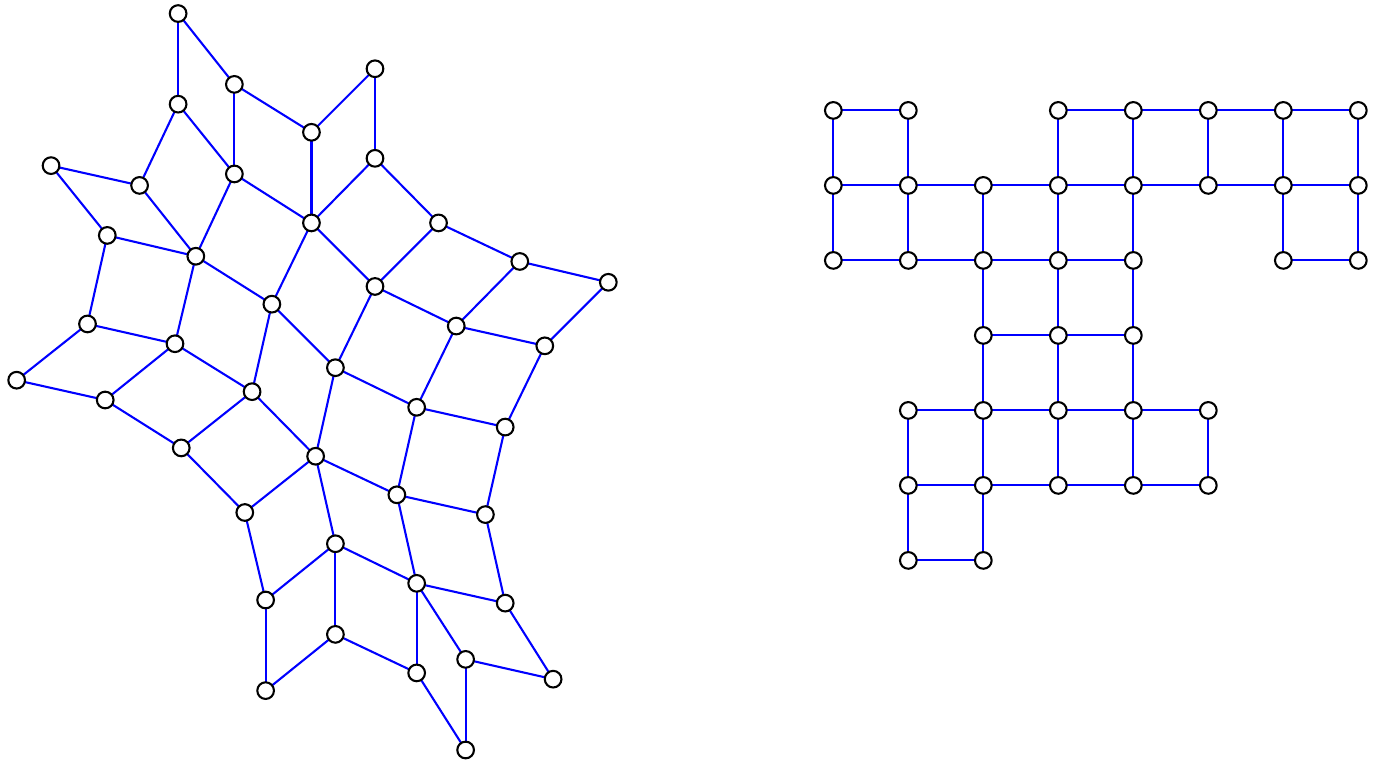}} \\
(a)&\hspace{0.75in}(b)&\hspace{0.75in}(c)
\end{tabular}
\caption{Three example of squaregraphs: (a) a squaregraph with three articulation points (also known as cut vertices); (b) a 2-connected squaregraph; (c) a polyomino.}
\label{fig:examples}
\end{figure}

Squaregraphs were first introduced in 1973 by Soltan, Zambitskii,
and Pris\v{a}caru \cite{SoZaPr} under the name ``graphs of class
$\mathcal K$''. In that paper several distance and structural
features of squaregraphs were investigated in order to solve the
median problem in this class of graphs; the authors showed that medians in squaregraphs can
be computed using the majority rule, just as in the case of trees.
Some properties of squaregraphs established in \cite{SoZaPr} are
also recycled in the present paper (see Lemmas \ref{4corners} and~\ref{strip} below).
The name ``squaregraph''
was coined in \cite{ChDrVa,ChFaVa}, where also related natural
classes of plane graphs were studied. Squaregraphs include and generalize
the \emph{polyominoes}~\cite{Go} formed by surrounding a region in a square
grid by a simple cycle (Figure~\ref{fig:examples}c). The
local constraint on the degrees of inner vertices entails that the
squaregraphs constitute one of the basic classes of face-regular
plane graphs with combinatorial nonpositive curvature \cite{BrHa}.
Squaregraphs have a plethora of metric and structural properties,
some of which follow from the fact (see below) that they are median
graphs and thus partial cubes (i.e., isometric subgraphs of
hypercubes). Median graphs and, more generally, partial cubes
represent two key classes of graphs in metric graph theory, which
occur in various areas and applications, and have been
(re-)discovered many times and under different guises
\cite{BaCh_survey,ChDrHa,EppFaOv}.

From an algorithmic point of view, squaregraphs constitute an
intermediate class of graphs forming a generalization of trees
and a special case of arbitrary
cube-free median graphs; like trees, squaregraphs admit linear or near-linear
time solutions to a number of problems that are not applicable to broader classes
of graphs. For example, the papers
\cite{ChDrVa,ChFaVa} present linear time algorithms for solving
diameter, center, and median problems in squaregraphs, which would
not necessarily carry over to all cube-free median graphs. More
recently, the embedding result of \cite{BaChEpp} was used in
\cite{ChFeGoVa}  to design a self-stabilizing distributed algorithm
for the median problem on even squaregraphs (i.e., squaregraphs in
which all inner vertices have even degrees). The paper
\cite{ChDrVa_CAT} presents a compact representation of some plane
graphs of combinatorial nonpositive curvature, containing
squaregraphs as a subclass, allowing one to answer distance and routing
queries in a fast way.

Our paper is organized as follows. The next section features some aspects of the algebraic
duality for (not necessarily finite) median graphs. In particular, every median graph
can be recovered from the traces of the convex splits (i.e., pairs of complementary halfspaces)
on any median-generating subset via a kind of conditional ``Hellyfication''.
 In Sections~\ref{sec:circular-split} and~\ref{sec:smallgen}, we will then see
that there are canonical choices for median-generating sets in the case of a finite squaregraph
that express their cyclic structure (e.g. by taking the boundary of the outer face).
In particular, these properties allow one to find a minimum-size median-generating set in polynomial time for finite squaregraphs,
whereas the corresponding problem for median graphs in general is NP-hard. In Section~\ref{sec:local-x}, squaregraphs are characterized among median graphs by forbidden configurations. This offers an ultimate generalization of squaregraphs to a particular class of infinite cube-free median graphs beyond plane graphs. Further, with these preliminaries out of the way, Section~\ref{sec:line-arrangements} introduces a class of
infinite squaregraphs and shows that they can be described in an equivalent way that does not presuppose a plane embedding.
As highlighted in the same section, the cyclic feature of squaregraphs is also reflected by
geometric duality, whereby any squaregraph can be interpreted as the dual of
a triangle-free chord diagram in the plane. In the Klein model of the hyperbolic plane
one can then extend this relationship to infinite squaregraphs as well.
 Section~\ref{sec:tree-products}
deals with isometric embedding into Cartesian products of trees and shows that five trees suffice for all squaregraphs.
We also display the chord diagram and its dual
associated with Ageev's example showing that there are finite squaregraphs that cannot be
embedded into the Cartesian product of just four trees. The proofs of all propositions and
the three theorems are deferred to Sections \ref{proofs:split-to-median}--\ref{proofs:tree-products}. The final section gives a brief outlook to the metric properties of squaregraphs
and their geometric extensions to injective spaces, which will be studied in a follow-up paper.

\section{From split systems to median graphs}
\label{sec:split-to-median}

\begin{figure}[t]
\centering\includegraphics[height=1.25in]{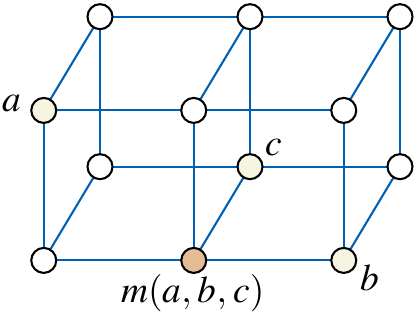}
\caption{The median of three vertices in a median graph.}
\label{fig:median}
\end{figure}

%\begin{figure}[t]
%\centering\includegraphics[width=2.5in]{median}
%\caption{The median $m$ of three vertices $a, b, c$ in a squaregraph. Three shortest paths between $a, b, c$
%passing through $m$ are highlighted in bold. Vertices of degree at most two are marked with shade and one
%boundary vertex of degree three is selected and marked by a cross.}
%\label{fig:median}
%\end{figure}

A \emph{median graph} \cite{MuSchr} is a graph in which every three
vertices $a,b,c$ have a unique \emph{median}, a vertex $m$ that
belongs to some shortest path between each two of the three vertices
$a$, $b$, and~$c$ (Figure~\ref{fig:median}). A \emph{median algebra} \cite{BaHe,Is},
generalizing the median operation in median graphs, consists of a
set of elements and a ternary commutative median operation
$m(a,b,c)$ on that set, such that $m(x,x,y)=x$ and $m(m(x,w,y),w,z)
= m(x,w,m(y,w,z))$ for all $w,x,y,z$. The \emph{median hull} of a subset
of a median algebra is the smallest median subalgebra including this subset.
In particular, if the entire median algebra is the median hull of a subset $S,$ then $S$ is
said to be a {\it median-generating} set of the algebra.  Later
(in Propositions~\ref{prop:median_hull} and
\ref{prop:median_hull_bis}) we shall investigate median-generating sets of squaregraphs.

As we now explain, median graphs may be characterized in terms of splits of their vertex sets.
A \emph{split} $\sigma=\{ A,B\}$ on a set $X$ is a partition of $X$
into two nonempty subsets $A$ and $B.$ A \emph{split system} on $X$ is any set of splits
on $X$. Two splits $\sigma_1=\{ A_1,B_1\}$ and
$\sigma_2=\{A_2,B_2\}$ are said to be \emph{incompatible} if all
four intersections $A_1\cap A_2, A_1\cap B_2, B_1\cap A_2,$ and
$B_1\cap B_2$ are nonempty, and are called \emph{compatible}
otherwise. In order to keep the set $X$ as
small as possible we may stipulate (whenever necessary) that
$\mathcal S$ (or the corresponding copair hypergraph)
\emph{separates the points,} i.e., for any two elements there exists
at least one split from $\mathcal S$ separating this pair. Split systems constitute a basic structure somewhat
analogous to hypergraphs. In fact, split systems can be turned into
\emph{copair hypergraphs} \cite{MuSchr,Mu}, also known as collections of
halfspaces \cite{ChDrHa}, which are hypergraphs in which the complement
of any hyperedge is also a hyperedge. In the context of geometric group theory, split
systems have been dubbed spaces with walls \cite{HaPa}. A split may be encoded
as a map from $X$ to $\{0,1\}$ (where the roles of $0$ and $1$ may be
interchanged), and split
systems may also be encoded by sets of these character
maps. These are the ``binary messages'' considered by
Isbell \cite{Is} in his duality theory for median algebras. In the
context of numerical taxonomy or phylogenetics one would speak of
binary data tables \cite{BaMaRi}.

\begin{figure}[t]
\begin{tabular}{cc}
\multicolumn{2}{c}{\includegraphics[width=3.5in]{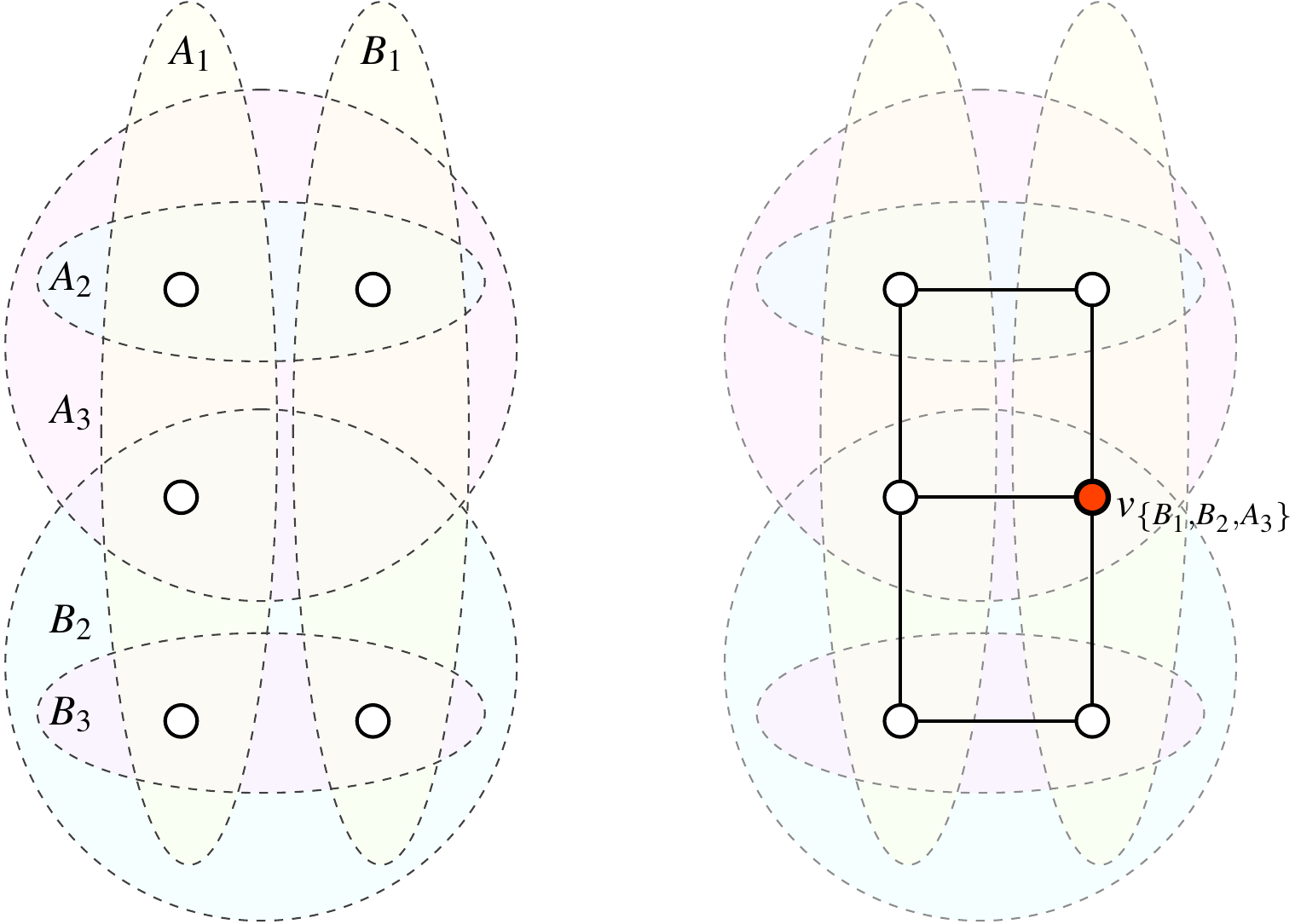}}\\
\hspace{0.65in}(a)&\hspace{0.85in}(b)\\
\end{tabular}
\caption{A split system $\{\{ A_1,B_1\}, \{ A_2,B_2\}, \{ A_3,B_3\}\}$ (left), and
its Hellyfication (right). The sets $B_1$, $B_2$, and $A_3$ form a maximal
pairwise intersecting family in the left split system, so in its Hellyfication
we add a new vertex $v_{\{B_1,B_2,A_3\}}$. The median graph formed by the resulting
Helly system is also shown.}
\label{fig:hellyfication}
\end{figure}

Every split system  $\mathcal S$  on a set $X$ has a natural
extension to a ``Helly'' split system on a larger set, via a construction performed
in \cite{Barth89, DrHuMo_Bune}. Here
``Helly'' refers to the Helly property of the associated copair
hypergraph $\mathcal H$: any pairwise intersecting family has a nonempty
intersection. The ``Hellyfication'' of a hypergraph ${\mathcal H}=
(X,{\mathcal E})$ extends the ground set $X$ and the hyperedges by
adjoining new elements that turn certain intersections of hyperedges
nonempty in order to gain the Helly property. Namely, for every
maximal pairwise intersecting set ${\mathcal F}$ of hyperedges with empty
intersection, add a new element $v_{\mathcal F}$ to $X$ and each
member of $\mathcal F$ (Figure~\ref{fig:hellyfication}). In the thus extended hypergraph $[{\mathcal
H}]$ with new ground set $[X]$ any two hyperedges intersect exactly when
their traces on $X$ intersect. Hence $[{\mathcal H}]$ is a Helly
hypergraph by construction, because every maximal set of pairwise
intersecting hyperedges has an element from $[X]$ in common. In the case
of a copair hypergraph associated with a split system on $X,$ the
maximal sets are transversals of the system, i.e., they comprise
exactly one part from each split. When $X$ is finite, the extended copair hypergraph
defines a median graph, with vertex set $[X]$ and with an edge between
any two vertices that are separated by exactly one split of the hypergraph; the proof of this fact is straightforward
from the main result of \cite{MuSchr} (for proofs, see also
\cite{Mu,BaHe}), which essentially involves the observation that a Helly copair
hypergraph determines its convex splits (the pairs of
complementary halfspaces of a median graph), and vice versa.

Every median graph yields a metric space via the shortest-path
length metric. This metric is determined by the unit-weighted convex
splits: the distance between any two vertices is the sum of the
weights of the splits separating them. If the weights are
arbitrarily positive, then one obtains a {\it median network}, which
is a discrete median metric space \cite{Ba_condorcet}. The notions
employed in the general metric context are then completely analogous
to the usual graph-theoretic notions: A subset $A$ of a metric space
$(X,d)$ is \emph{convex} if the (metric) \emph{interval} $I(u,v)=\{
x\in X: d(u,x)+d(x,v)=d(u,v)\}$ between any two points $u$ and $v$
of $A$ lies entirely in $A.$ The {\it convex hull} of a subset
$B$ of $X$  is the smallest convex set containing $B.$  A subset $Y$ of $X$ is \emph{gated} if
for every point $x\in X$ there exists a (unique) point $x'\in Y$
(the \emph{gate} for $x$ in $Y$) such that $x'\in I(x,y)$ for all
$y\in Y$ (cf. \cite{DrSch}). A split $\sigma=\{ A,B\}$ on $X$ is
called a \emph{gated split} (resp., \emph{convex split}) if the sets
$A$ and $B$ are gated (resp., convex); the parts $A$ and $B$  of a
convex or gated split $\{ A,B\}$ are called \emph{halfspaces}.
In median graphs and networks, all convex sets are gated \cite{VdV}. Moreover, it
is well known that in a median graph $G=(X,E)$ the splits
$\sigma(uv)=\{ W(u,v),W(v,u)\}$ separating the edges $uv\in E$ are
convex and therefore gated \cite{BaCh_survey,Mu,VdV}, where
$W(u,v)=\{ x\in X: d(u,x)<d(v,x)\}$ and $W(v,u)=X\setminus W(u,v).$ We denote
by ${\mathcal S}(G)$ the resulting collection of convex splits of a
median graph $G$ and by ${\mathcal H}(G)$ the corresponding copair
hypergraph of halfspaces  (halfspace hypergraph, for short).

The preceding observations extend to the discrete
infinite case. A helpful example in this context is given by the split system on the grid
$\mathbb{Z}^2$, which we would like to view as an example of an infinite squaregraph. The system
consists of the splits $\{ A_k,B_k\}$ and $\{ A_k',B_k'\}$, where $A_k=\{(x_1,x_2)\mid x_1\le k\}$ and
$A_k'=\{(x_1,x_2)\mid x_2\le k\}$ $(k\in \mathbb{Z})$. Now take the trace ${\mathcal S}={\mathcal S}(X)$
of this system on any subset $X$ of ${\mathbb Z}^2$ that meets all lines $A_{k+1}\cap B_k$
and $A'_{k+1}\cap B'_k$ and all quadrants $A_k\cap A'_k, A_k\cap B'_k, B_k\cap A'_k,$ and
$B_k\cap B'_k$ ($k\in {\mathbb Z}).$ We would like to recover $\mathbb{Z}^2$ as the Hellyfication
of ${\mathcal S}$, however there is a complication:
while certain maximal sets of pairwise intersecting split parts
correspond to the vertices of the grid, others would not, such as the set
$\{ A_k\cap X \mid k\in {\mathbb Z}\}\cup \{ A'_k\cap X\mid k\in \mathbb{Z}\}$. Thus,
to generate $\mathbb{Z}^2$ from ${\mathcal S}$, we cannot create new vertices for every
maximal pairwise-intersecting family (transversal of $\mathcal S$), but must keep some
and discard others. The key
feature discriminating between these two types of transversal
is ``anchoring'', to be described next.

We say that a split system $\mathcal S$ or a hypergraph $\mathcal H$ on a not
necessarily finite set $X$ is \emph{discrete} if any two
points of $X$  are separated by only finitely many splits or hyperedges, respectively; separation,
of course, means that the two points  are not
included in the same split part or hyperedge. A collection $\mathcal D$ of hyperedges in a  hypergraph
${\mathcal H}=(X,{\mathcal E})$ is said to be {\it anchored} if there exists a finite subset $Z$ of $X$
that has a nonempty intersection with every member of $\mathcal D$. In a discrete hypergraph, $\mathcal D$ is anchored exactly
when every point $x\in X$ belongs to all but finitely many
members of $\mathcal D$: if this is the case, the finite set $Z$ may be found by choosing any $x$
and one additional element from each set not containing $x$, while if $Z$ exists then for any $x$
there are only finitely many hyperedges separating
$\{x\}\cup Z$ by discreteness and all other hyperedges must contain $x$.
The Helly property for an infinite hypergraph is usually understood to require
that every {\it finite} pairwise intersecting family of hyperedges has a nonempty intersection. The halfspace
hypergraph of a median graph, however, satisfies a slightly stronger Helly property as we will see next.
We call a hypergraph {\it $^*$Helly} if every anchored family of hyperedges that intersect in pairs has a nonempty intersection.

\begin{proposition}
\label{prop:Helly*} A discrete hypergraph ${\mathcal H}=(X,{\mathcal E})$ is $^*$Helly if
and only if for every triplet $u,v,w$ of points the intersection of the hyperedges containing at least two of
$u,v,w$ is nonempty.
\end{proposition}

In the finite case this is a classical result on Helly hypergraphs~\cite{BeDu}; see also \cite[Corollary to Theorem 10]{Be}.
The convex sets of a median graph $G $ satisfy the condition of the preceding proposition
because the median of the three points in question guarantees that the intersection is not empty:

\begin{corollary} \label{Helly_convex} The hypergraph ${\mathcal C}(G)$ of all convex sets in a median graph
$G=(V,E)$ has the $^*$Helly property.
\end{corollary}

Note that the hypergraph ${\mathcal C}(G)$ is not discrete in general but the halfspace hypergraph
${\mathcal H}={\mathcal H}(G)$ is. For any anchored and pairwise intersecting family $\mathcal D$ of convex sets
consider the family $\mathcal F$ of all halfspaces that include members of ${\mathcal C}(G).$ The median of three vertices
$u,v,w$ belongs to every halfspace that contains at least two of them. Therefore, by
Proposition~\ref{prop:Helly*}, we have $\emptyset\ne\bigcap {\mathcal F}=\bigcap {\mathcal D},$ because every
convex set is an intersection of halfspaces. Then the finite case of
Corollary \ref{Helly_convex}, formulated by Evans \cite[Theorem 3.5]{Ev} as the ``Chinese Reminder Theorem", could also be
viewed as a direct consequence of \cite{BeDu}. This result was later rediscovered by \cite{Ro} (cf. Theorem 2.11 of \cite{Gu}).

\begin{proposition}
\label{prop:Mulder_Schrijver}
The following statements are equivalent for a discrete hypergraph ${\mathcal H}=(X,{\mathcal E}):$

\begin{itemize}
\item[(i)] $\mathcal  H$ is a maximal $^*$Helly copair hypergraph;
\item[(ii)]  $\mathcal H$ is a $^*$Helly copair hypergraph separating the points;
\item[(iii)]  $\mathcal H$ is the halfspace hypergraph of some median graph $G$ with vertex set $X.$
\end{itemize}
\end{proposition}

The preceding proposition generalizes the Mulder and Schrijver theorem \cite{MuSchr} from the finite to the infinite case.
Now, as in the finite case, we can expand a discrete copair hypergraph to one that is $^*$Helly.

\begin{proposition}\label{prop:Nica}\cite{Ni}
Let ${\mathcal H}=(X,{\mathcal E})$ be a discrete copair hypergraph separating the points.
Every anchored pairwise intersecting subset $\mathcal D$ of $\mathcal E$ can be expanded to a maximal anchored pairwise
intersecting subset $\mathcal F$ of $\mathcal E$; then $\mathcal F$ will contain one set from each complementary pair of sets in $\mathcal H$.
For every maximal anchored pairwise intersecting subset $\mathcal F$ of $\mathcal E$ that has an empty intersection, create a new element
$v_{\mathcal F}$, and include $v_{\mathcal F}$ as a member of each set in $\mathcal F$.
Then the resulting expanded hypergraph
$[{\mathcal H}]=([X],[{\mathcal E}])$ constitutes the halfspace hypergraph of a median graph with vertex set $[X]$ such that
$[X]$ is the median hull of $X$ and the trace of $[{\mathcal E}]$ on $X$ equals $\mathcal E$.
\end{proposition}

We refer to the extension of a discrete copair hypergraph $\mathcal H$ to $[{\mathcal H}]$ as the \emph{$^*$Hellyfication} of
$\mathcal H$. The maximal anchored pairwise intersecting collections were called almost
principal ultrafilters by Nica \cite{Ni}. His Theorem 4.1 (together with some information from its proof)
is rephrased here as Proposition~\ref{prop:Nica},  with a short proof
based on Proposition~\ref{prop:Mulder_Schrijver}. In the finite case, this result was well known; see \cite{DrHuMo_Bune}.

\section{From 2-compatible circular split systems to squaregraphs}
\label{sec:circular-split}

\begin{figure}[t]
\centering\includegraphics{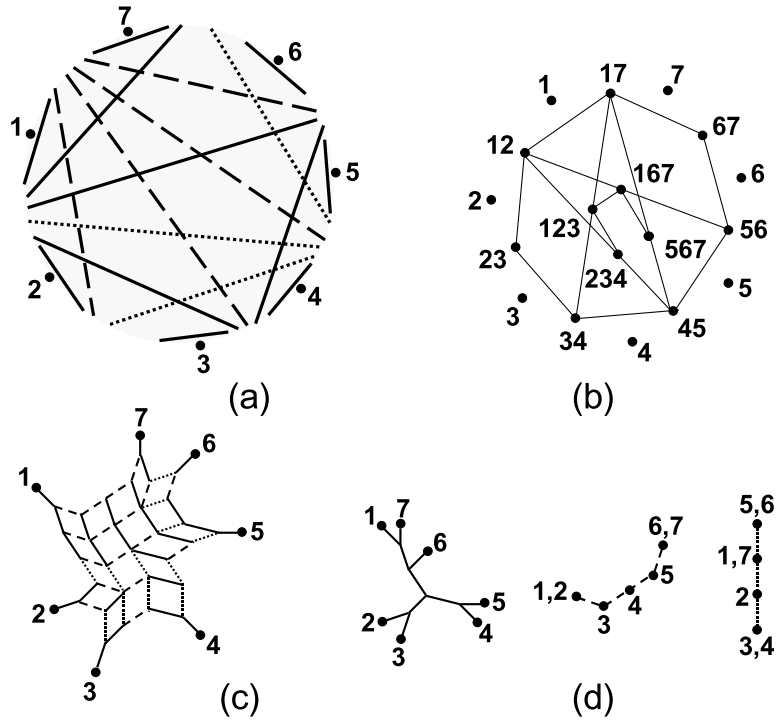}
\caption{(a) 3-Colored
system of 18 chords of a circle that separate points 1-7, where
chords of the same color (either unbroken, or broken, or stippled)
do not intersect, and (b) the corresponding circle graph, where
vertices correspond to the chords that are encoded by the points on
the smaller arcs they each bound. (c) Squaregraph associated with
the split system depicted in (a); every edge is marked with the
color of the split it crosses. (d) Trees supported by each color
class of splits/edges; the squaregraph canonically embeds into the
Cartesian product of these three trees.}
\label{fig:split-7-18}
\end{figure}

A split system $\mathcal S$ is called \emph{2-compatible}~\cite{DrKoMo}
if it does not contain any  three
pairwise incompatible
splits. A split system $\mathcal S$ on a cycle with vertex set $X$
is said to be \emph{circular} if for each split $\sigma=\{ A,B\}$ of
$\mathcal S$, the parts $A$ and $B$ constitute complementary paths
of this cycle. Circular split systems arise as split systems with
certain extremal properties \cite{BaDr,DrKoMo}. To give an example,
Figure~\ref{fig:split-7-18}a displays a circle with seven points highlighted forming the
set $X = \{ 1,2,3,4,5,6,7\}.$ The 18 chords indicate splits on $ X$
forming the system $\mathcal S.$ In this example no triplet of
splits would be pairwise incompatible. As a consequence, the median
graph with vertex set $[X]$ is composed of 4-cycles and bridges
(Figure~\ref{fig:split-7-18}c). Note that  18 is the largest possible
number of 2-compatible splits on a set with $n=7$ elements~\cite{DrKoMo}.

Given a system $\mathcal S$ of splits on $X,$ the {\it
incompatibility graph} Inc$({\mathcal S})$ of $\mathcal S$ has the
splits as vertices and pairs of incompatible splits as edges. Inc$({\mathcal S})$ is
triangle-free exactly when $\mathcal S$ is 2-compatible. In the
case of a circular split system, the incompatibility graph can be
regarded as the intersection graph of chords of a circle, which is
referred to as a \emph{circle graph}~\cite{JeTo}. The incompatibility graph of the split system depicted in
Figure~\ref{fig:split-7-18}a is shown in Figure~\ref{fig:split-7-18}b.
%If, in addition
%$\mathcal S$ is 2-compatible, then Inc$({\mathcal S})$ is
%triangle-free, and vice versa.

The boundary cycle of a finite 2-connected squaregraph plays a key
role in a succinct median description of squaregraphs, which is
expressed by the next propositions. An orientation of a cycle $X$ yields a ternary relation $\beta$ on $X$ where
$\beta (u,v,w)$ expresses that the directed path from $u$ to $w$ passes through $v.$ This relation is total,
asymmetric, and transitive, which can be formulated in terms of Huntington's axioms \cite{Hu}: for any four points $u, v, w, x$ of $X,$

		\begin{itemize}
         \item[] if $u,v,w$ are distinct, then $\beta (u,v,w)$ or $\beta (w,v,u),$
		\item[] $\beta (u,v,w)$ and $\beta (w,v,u)$ is impossible,
		\item[] $\beta(u,v,w)$ implies $\beta (v,w,u),$
		\item[] $\beta (u,v,w )$ and $\beta (u,w,x)$ imply $\beta (u,v,x).$
\end{itemize}

A ternary relation $\beta$ on an arbitrary set $X$ satisfying these axioms is called a {\it total cyclic order.}
It follows from this definition that only triplets of distinct points can be in the relation $\beta$ and that the reverse
(alias opposite) relation $\beta^{op}$ defined by $\beta^{op}(u,v,w)$ exactly when $\beta (w,v,u)$ is also a total cyclic
order. We say that a nonempty proper subset $A$ of $X$ is an arc if there are no four distinct points  $u,v\in A$  and
$x,y\in X\setminus A$  such that  $\beta (u,x,v), \beta (x,v,y), \beta (v,y,u),$ and $\beta (y,u,x)$  hold. If $A$ is an arc,
then so is $X\setminus A$ as well. In the finite case, total cyclic orders are just the orientations of cycles.
A split $\{ A,B\}$ is {\it circular} with respect to the total cyclic order
if its parts are arcs. A split system on $X$ is said to be {\it circular} if all its members are circular with respect to some
total cyclic order on $X.$ Thus, this definition extends the notion of circular split systems to the infinite case. To give a preliminary
definition, we say that
a median algebra $M$ is circular if the system of convex splits restricted to some median-generating
set $X$ of $M$ is circular. For example,
the Cartesian product of a finite path with a ray (one-way infinite path) has an infinite boundary (from which one can retrieve
the entire partial grid) that bears a total cyclic order. In a finite 2-connected squaregraph every convex split has a circular
trace on the oriented boundary cycle. The cyclic order
is uniquely determined
up to reversal by the requirement that the traces of the convex splits be circular; see the next proposition.
In contrast, for a finite tree $T,$ total cyclic orderings exist in abundance for which  the convex splits become
circular \cite{SeSt}. In fact, any planar representation of $T$ in the plane together with an arrangement of
non-crossing (pseudo-)lines such that the lines cross the edges of $T$ in a one-to-one manner (referred to
as a ``Meacham egg''; see Fig. 3 of \cite{BaDr86}) determines such a cyclic ordering, and vice versa.

\begin{proposition}
\label{prop:unique_cyclic_order}
Let $G=(V,E)$ be a finite squaregraph with boundary vertex set $C$.
Then the convex splits of $G$ restrict to circular splits on $C$. The corresponding cyclic order on $C$ is unique (up to reversal) if and only if either $\#V < 4$ or $G$ is 2-connected.
\end{proposition}

\section{Small generating sets}
\label{sec:smallgen}

In the next section we will define a class of infinite squaregraphs, including the integer grid $\mathbb{Z}^2$, which may be completely devoid of a boundary (as $\mathbb{Z}^2$ is) or may at least be deficient in boundary vertices. This deficiency will manifest itself in the property that the median hull of the boundary is not the entire squaregraph. In contrast, the boundary of a finite squaregraph generates the entire squaregraph  as its median hull. In fact, an even smaller subset than the boundary suffices.

\begin{proposition}
\label{prop:median_hull}
Every finite squaregraph $G=(V,E)$ has at least
$\min\{ 4 , \#E\}$ vertices of degrees one or two. In a finite squaregraph, the vertices of degree at most three are exactly the articulation points of degree two or three and the endpoints of maximal convex paths. Let $X$ be a subset of $V$ that includes all vertices of degrees one and two in $G$ and includes at least one vertex of every maximal convex path $P$ of $G$ that does not pass through an articulation point of $G$. Then  every vertex of $G$ is the median of three vertices from $X$.
\end{proposition}

In general, a finite squaregraph may not be the median hull of its vertices of degrees one or two: the domino $K_2\Box K_{1,2}$ constitutes the smallest pertinent example. Even when the median hull of the vertices of degrees at most two is the entire squaregraph, not every vertex needs to be the median of three vertices of degree at most two. Take, for instance,  the vertex of degree 6 in the graph of Figure \ref{fig:examples}a, where a second step of iterating the median operation is necessary to obtain this vertex. The next proposition fully characterizes the median-generating sets. A convex path $P$ is said to be an {\it inner line} of a finite squaregraph $G$ if its two endpoints are boundary vertices of degree 3 and all its inner vertices are inner vertices of degree 4 in $G.$ The endpoints of an inner line $P$ cannot be articulation points of $G$, and $P$ is necessarily a maximal convex path.

%\begin{proposition}
%\label{prop:median_hull_bis} A boundary vertex $v$ of a finite squaregraph $G=(V,E)$ is either the median of three vertices of degree at most two or it bounds a maximal convex path $P$ of some block $B$ of $G$ such that $P$ equals the intersection of two distinct halfspaces $H_1$ and $H_2$ of $G.$ Consequently, a subset $X$ of $V$ median-generates $G$ if and only if $X$ includes all vertices of degrees at most two and meets every path $P$ of the preceding kind, in which case every vertex of $G$ is the median of three medians of three vertices each from $X.$ The paths $P$ considered in Proposition \ref{prop:median_hull_bis} are exactly the maximal convex paths of blocks of $G$ for which every interior vertex has degree 4 and the bounding vertices have degree 3 and are not articulation points of $G.$
%\end{proposition}

\begin{proposition}
\label{prop:median_hull_bis} The inner lines of $G$ are exactly the paths that lie entirely in 	2-connected blocks and equal the intersection of two halfspaces of $G.$ A boundary vertex $v$ of a
finite squaregraph $G=(V,E)$ is either the median of three vertices of degree at most two or it bounds an inner line. Consequently, a subset $X$ of $V$ median-generates $G$ if and only if $X$ includes all vertices of degree at most two and meets every inner line, in which case 	every vertex of $G$ is the median of three vertices which are each medians of three vertices from $X.$
\end{proposition}

This result can be viewed as a tightening in case of squaregraphs of the characterization of median closures in general median algebras and median graphs given in \cite{Ba_gen,Berg,VdV}:  a vertex $v$ of a median graph $G$ belongs to the median closure of a set $X$ if and only if whenever $v$ belongs to the intersection of two halfspaces $H_1$ and $H_2$ of $G$ this intersection also contains a vertex of $X.$ Therefore, a subset $X$ is a median-generating set of $G$ if and only if any pair of halfspaces that intersect in $G$ also intersect in $X.$ Hence, if $X$ median-generates $G,$ then the incompatibility graph Inc$({\mathcal S}(G))$	of convex splits of $G$ is isomorphic to the incompatibility graph Inc$({\mathcal S}(G)|_X)$ of their traces on $X.$

Proposition \ref{prop:median_hull_bis} allows us to design a simple polynomial algorithm for computing a minimum median-generating set of a finite squaregraph $G.$ Since every degree 4 inner vertex of $G$ lies on at most two inner lines, the minimum number of such vertices to be added to a median-generating set can be determined by finding a maximum matching in the cross graph of inner lines: each inner line of $G$ corresponds to a node of the cross graph, and two nodes are linked by an edge if the corresponding inner lines share one common vertex. If $s$ is the total number of inner lines and $h$ is the size of a maximum matching in the cross graph, then $s-h$ vertices need to be added to the set of vertices of degree less than three in order to form a minimum-size median-generating set. In contrast, the task of computing a median-generating set of minimum size is NP-hard for arbitrary median graphs as the following theorem shows.

\begin{theorem}
\label{theorem:npc-gen}
The problem of deciding whether a given graph $G$ has a median-generating set $X$ of size at most a given number $g$ is NP-complete for finite median graphs not containing any 5-cube
but can be solved in polynomial time for finite squaregraphs.
\end{theorem}

There is a metric concept of generation (producing a kind of injective envelope) that makes do only with the vertices of degree less than three. Given any subset $X$ of vertices from a finite bipartite graph $G$ with shortest-path metric $d$, we say that a map $f$ from $X$ to the nonnegative integers $\mathbb{N}$ is a \emph{parity-integer metric form} on $X$ if
$$f(x) + f(y) - d(x,y) \ge 0 \mbox{ and  } f(x) + f(y) + d(x,y) \mbox{  is even  for all } x,y \in X.$$
Such a map $f$ is said to be minimal if any parity-integer metric form $g$ on $X$ that is pointwise below $f$ equals $f$. Every vertex $u$ of the bipartite graph $G$ is associated with the minimal parity-integer metric form $d_u$ on $X$ defined by $d_u(v) = d(u,v)$. $G$ is called an \emph{absolute retract of bipartite graphs} if $G$ can be retracted from every bipartite supergraph $H$ in which $G$ is isometric~\cite{BaDr86}.

\begin{proposition}
\label{prop:injective-hull}
In any finite cube-free median graph $G = (V,E)$ with shortest-path metric $d$, every vertex is uniquely determined by its distances to the vertices of degrees one or two. Specifically, $(V,d)$ is isomorphic to the subspace of the $L_\infty$ space $\mathbb{N}^X$ that consists of the minimal parity-integer metric forms on the metric subspace $(X,d|_X)$ of $(V,d)$ comprising the vertices of degrees one or two in $G$. Thus, $G$ is the smallest absolute retract of bipartite graphs extending the metric subspace $(X,d|_X)$.
\end{proposition}

A coarser concept of generation is the convex closure yielding the convex hull of any subset $Y$ of the vertex set. Trivially, every finite squaregraph $G$ is the convex hull of its vertices with degrees one or two. But in contrast to the previous generation concepts one can dispense with some of the vertices of degree two. The minimal number of vertices needed to generate the entire squaregraph $G$ as the convex hull has been termed the {\it hull number} $h(G)$. This evidently equals the minimum cardinality of a set of vertices that is incident with every minimal halfspace of $G$ \cite{Mu_exp}. The hull number of any median graph $G$ equals two exactly when $G$ is the covering graph of a finite distributive lattice $L$~\cite{BirKis-BotAMS-47}; in the case of a squaregraph, $L$ is embedded in a product of two chains. The minimal halfspaces of a finite squaregraph $G$ are of two types: any degree-one vertex is itself a minimal halfspace, as is each (maximal) convex path on the boundary of $G$ that has two degree-two endpoints and that does not contain any degree-four vertex or articulation point of $G.$ Note that the intersection graph $I(G)$  of the minimal halfspaces of $G$ constitutes an induced subgraph of the incompatibility graph Inc$({\mathcal S}(G))$ of convex splits.
Two intersecting minimal halfspaces can only share one common degree-two endpoint, and this shared vertex then corresponds to the edge linking the two minimal halfspaces in the intersection graph. Hence $I(G)$ is an even or odd cycle exactly when $G$ is 2-connected and all boundary vertices have degrees two or three. In all other cases, this intersection graph is a disjoint union of paths. The edges of this incompatibility graph correspond to vertices of degree two in $G,$  whereas the singleton components either comprise a single vertex of degree one in $G$ or represent some boundary path between two vertices of degree two. Therefore the hull number $h(G)$ equals the minimum number of edges and isolated vertices from $I(G)$ that jointly cover all vertices of this graph. Closely related to the hull number is another parameter for $G.$ The largest size $k$ of an independent set in the intersection graph $I(G)$ of all minimal halfspaces of $G$ is referred to as the {\it star-contraction number} $s(G),$ because the contraction of all halfspaces not belonging to this maximum independent set produces a star with $k$ leaves as a median homomorphic image of $G$ \cite{Mu_exp}. Summarizing this discussion, we obtain the following result:

\begin{proposition}
\label{prop:convex-hull}
The hull number $h(G)$ and star-contraction number $s(G)$ of a finite squaregraph G can be read off from the connected components (a single cycle or several paths) of the intersection graph $I(G)$ of minimal halfspaces of $G:$ in the case of a single cycle component, $h(G)$ equals half the size of the cycle rounded up to the next integer and $s(G)$ equals half the size of the cycle rounded down to the next integer, and otherwise, each path component of $I(G)$ contributes to both $h(G)$ and $s(G)$ half the number of path vertices rounded up to the next integer. Consequently, $h(G)=s(G)+1$ when $I(G)$ is an odd cycle, and in all other cases, $h(G)=s(G).$ In particular, $h(G)=s(G)$ if $G$ is an even squaregraph.
\end{proposition}

The relationship between  $h(G)$ and $s(G)$ provided by Proposition \ref{prop:convex-hull} cannot be generalized to larger classes of median graphs: there exist cube-free median graphs
for which $h(G)-s(G)$ can be arbitrarily large \cite{BaCh_acyclic}. Note also that the equality $h(G)=s(G)$ does not characterize the even squaregraphs:  to find squaregraphs $G$ in which
all inner vertices have odd degrees while $I(G)$ is an even cycle,  just take a sneak preview to Figure \ref{fig:45tess} in the next section.

Concepts related to hull number and star-contraction number refer to the compatibility graph of the convex splits of $G.$ The clique number of the latter was denoted in \cite{BaCh_acyclic} by $t(G),$ whereas the chromatic number was denoted by $c(G).$ Then $t(G)$ is the maximum number of edges in any median-homomorphic tree image, whereas $c(G)$ is the minimum number of squares and single edges that together meet all halfspaces of $G.$ In the simplex graph $G$ of a 5-cycle one has $c(G)=t(G)+1=3.$ Now, there are squaregraphs with an arbitrarily large number of convex subgraphs isomorphic to this simplex graph such that each halfspace meets at most one of those distinguished subgraphs (to find such graphs, see again  Figure \ref{fig:45tess} in the next section). Therefore $c(G)-t(G)$ can attain any nonnegative number for finite squaregraphs $G,$ although even squaregraphs $G$ are again characterized by equality $c(G)=t(G).$
Passing from the compatibility graph to the incompatibility graph then changes the picture considerably. The clique number of the incompatibility graph of the convex splits of $G$ is trivially 2 unless $G$ is a tree (where it is 1), but the chromatic number will then determine how many tree factors are needed for coordinatizing $G.$ This is exactly the problem that we will solve in
Section \ref{sec:tree-products}.

\section{Local characterization of squaregraphs}
\label{sec:local-x}

Finite squaregraphs can be characterized among finite median graphs in terms of forbidden subgraphs (which can be assumed
to be either induced or isometric or convex, respectively), without reference to an embedding into the plane.
The forbidden configurations are then reflected by halfspace patterns as well.
Given a discrete copair hypergraph ${\mathcal H}= (X,{\mathcal E}),$
a \emph{suspended cycle} in $\mathcal H$ consists of $k+1$ hyperedges
$A_0, A_1,\ldots, A_{k-1}, A_k=A_0$ $(k > 3)$ and $D$ such that (i)
$A_i$ and $A_j$ intersect exactly when $|j-i|=1$ and (ii) $D$ is
included in all complementary hyperedges $X\setminus A_i$ $(i=0,\ldots, k-1).$ A
minimal set $X$ realizing the $k+1$ splits involving a suspended
cycle has $k+1$ elements. The median graph obtained from the
Hellyfication of the corresponding copair hypergraph is then a
\emph{suspended cogwheel}; see Figure~\ref{fig:forbidden}. A suspended cogwheel is thus
formed by a {\it cogwheel} \cite{KlKo} (alias bipartite wheel) plus
a pendant vertex adjacent to the center of the wheel. We dot not regard the cube minus a vertex as a cogwheel here.
In particular, a $k$-{\it cogwheel} then consists of a central vertex and a cycle of length 	$2k > 6$ such that every
second vertex is adjacent to the center. A $k$-cogwheel 	is said to be {\it even} or {\it odd,} respectively,
depending on the parity of $k.$

\begin{proposition}
\label{prop:forbidden}
For a finite graph $G=(V,E),$ the following
statements are equivalent:
\begin{itemize}
\item[(i)]   $G$ has a plane embedding as a squaregraph;
\item[(ii)]   $G$ is a median graph such that $G$ does not contain any of
the following graphs as induced subgraphs as induced subgraphs
(or isometric subgraphs, or convex subgraphs, respectively): the cube,
$K_2\Box K_{1,3}$, and suspended cogwheels;
\item[(iii)]   $G$ is the median graph associated with the
Hellyfication of a  copair hypergraph $\mathcal H$ on some finite
set $X$ which has no triplet of properly intersecting hyperedges, no halfspace that
properly intersects three pairwise disjoint halfspaces,  and no
suspended cycle;
\item[(iv)]    $G$ is a cube-free median graph that is the
median hull of some subset $X$ of the vertex set $V$ such that the
trace of the system of convex splits of $G$ on $X$ is circular.
\end{itemize}
\end{proposition}

\begin{figure}[t]
\centering\includegraphics{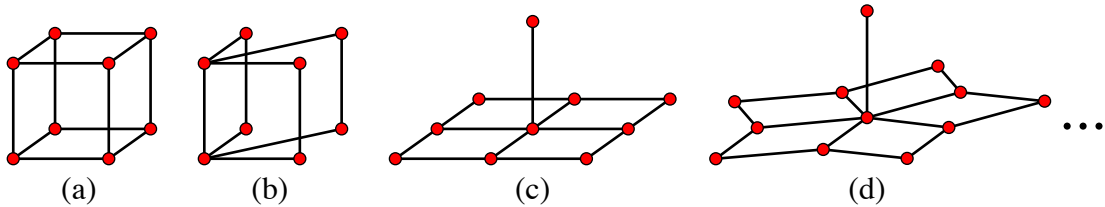}
\caption{Forbidden induced subgraphs of squaregraphs: (a) cube; (b) $K_2\Box K_{1,3}$; (c)\&(d)~the first two suspended cogwheels.}
\label{fig:forbidden}
\end{figure}

This structural characterization of squaregraphs can be generalized
to infinite median graphs in a straightforward way.   An infinite plane graph
is called an {\it infinite  squaregraph} if every finite convex subgraph
consitutes a finite squaregraph.  In a finite
squaregraph one has a well-defined boundary cycle. In the infinite
case, however, we have to resort to a substitute, the ``virtual
boundary", where rays come into play.
Let $G$ be an infinite graph such that the convex hull of every
finite set in $G$ is a finite squaregraph.  Then $G$ is necessarily
a cube-free median graph, in which every edge $uv$ gives rise to a zone $Z(uv)$
just as in the case of a finite squaregraph: $Z(uv)$ is the ladder subgraph
induced by the vertices incident with the edges from the equivalence class
$\Theta(uv)$ (see Section \ref{proofs::circular-split}).
Then $P(u,v)=Z(uv)\cap W(u,v)$ and $P(v,u)=Z(uv)\cap W(v,u)$ are convex paths,
which are both  either finite or one-way or two-way infinite. Now, extend the
vertex set $V$ of $G$ by adding virtual endpoints to all rays that border a zone.
The same endpoint is added to all zonal rays that are comparable with respect to
inclusion. The {\it virtual boundary} of $G$ consists of all endpoints of
border paths of zones and virtual endpoints of zonal rays of $G.$ We call a graph
$G$ {\it circular} if the
traces of convex splits of $G$ on the virtual boundary of $G$ form a circular split system
(whenever a ray is included in a halfspace $H$ of $G,$ then its endpoint is included in the trace of
$H$ on the virtual boundary).

\begin{proposition}
\label{prop:forbidden-infinite}
For an infinite median graph $G,$  the following
statements are equivalent:
\begin{itemize}
\item[(i)]   the convex hull of every finite set in $G$ is a finite squaregraph;
\item[(ii)] $G$ is the directed union of convex subgraphs that are
finite squaregraphs;
\item[(iii)]  $G$ is cube-free such that $K_2\Box K_{1,3}$ and the suspended
cogwheel are not induced (or isometric or convex) subgraphs of $G$;
\item[(iv)] $G$ has no triplet of properly intersecting halfspaces,
no halfspace that properly intersects three pairwise disjoint
halfspaces, and no cycle of properly intersecting halfspaces
included in another halfspace;
\item[(v)]   $G$ is cube-free and circular.
\end{itemize}
\end{proposition}

%Note that the conditions of this proposition are less restrictive than those of
%Theorem~\ref{theorem:infinite} below; for instance, an infinite star $K_{1,\infty}$
%satisfies the conditions of this proposition but not those of Theorem~\ref{theorem:infinite}.

\section{Infinite squaregraphs and geometric duality}
\label{sec:line-arrangements}

\begin{figure}[b]
\centering\includegraphics[scale=0.57]{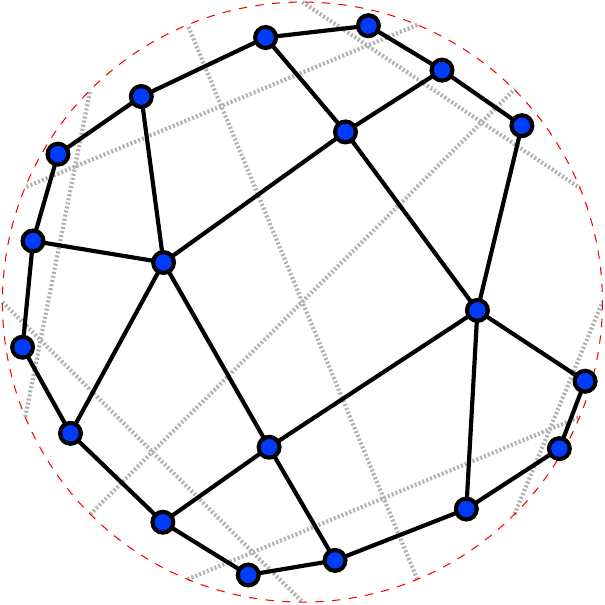}
\caption{A chord diagram and its planar dual, a squaregraph.}
\label{fig:dual}
\end{figure}

There is yet another natural way of interpreting the representation
of a squaregraph by a diagram of chords within a unit disk of the
Euclidean plane. Consider any finite arrangement of (straight-line)
chords of a unit circle such that no two chords intersect on the
circle and no three chords intersect pairwise. The resulting
configuration is then called a {\it triangle-free chord diagram}.
For further background and references concerning arbitrary chord
diagrams, see \cite{Sa02}. Alternatively, one can view a chord
diagram as a certain cyclic double permutation of the numbers from 0
to $n-1$: label $2n$ distinct points on the circle arbitrarily by
using each number exactly twice; a pair of points with the same
label is then connected by a chord. This labeling yields a
triangle-free chord diagram exactly when no three numbers $i,j,k$
are traversed along the oriented circle in the order $i,j,k,i,j,k.$
Every (finite) triangle-free chord diagram yields a squaregraph by
taking the planar dual of the map bounded by the chorded circle; see
Figure~\ref{fig:dual}.
Indeed, every region of the map not bounded by an arc of the circle
has degree at least 4 in the dual graph because the chord diagram is
triangle-free; moreover, all faces of the planar dual are 4-cycles
since every intersection point of the chords and the circle belongs
to either exactly two chords or one chord and the circle. Note that
the interpretation of a chorded circle in the preceding section
simply refers to the boundary regions of the planar map, where the
chords then represent the splits for the boundary regions (which are
represented in Figure~\ref{fig:forbidden} by points on the circle). Therefore the
planar dual simply reflects the algebraic dual \cite{Is} of the
2-compatible circular split system.

\begin{figure}[t]
\centering\includegraphics[height=3in]{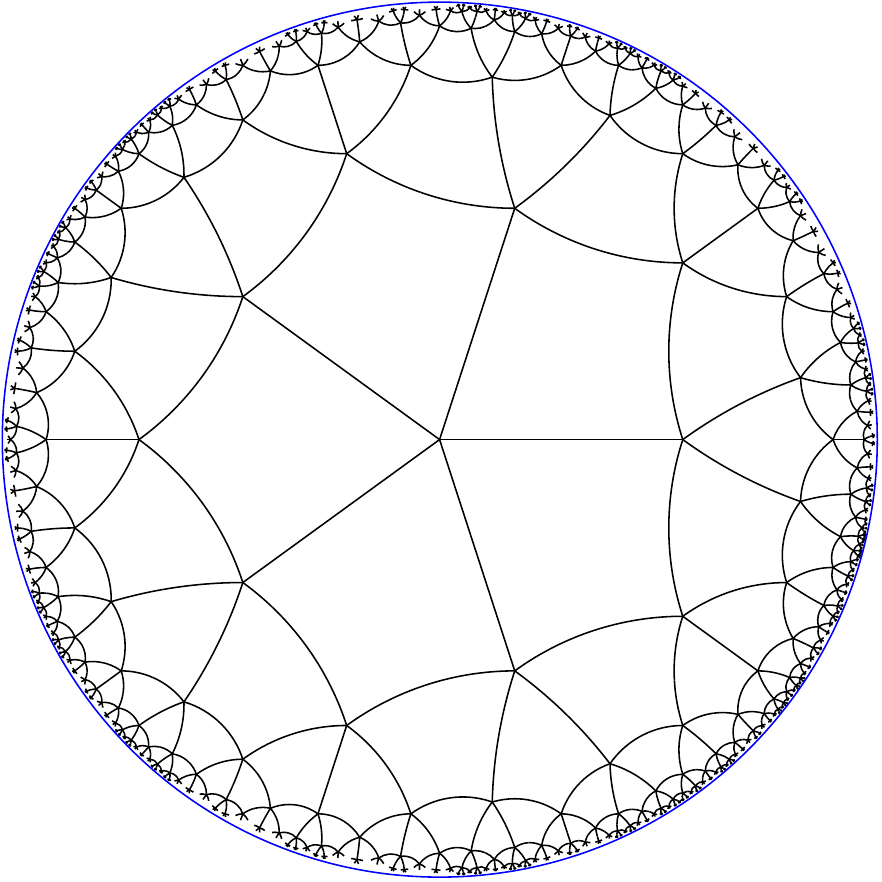}
\caption{The $\{4,5\}$ tesselation of the hyperbolic plane, reproduced from Figure 10 of \cite{Epp_orbi}.
Image produced using a Java applet written by Don Hatch, http://www.plunk.org/$\sim$hatch/HyperbolicApplet/.}
\label{fig:45tess}
\end{figure}

For finding a suitable extension to the infinite case, it is helpful
to view a chord diagram as a set of lines in the Klein model of the
hyperbolic plane. So, a finite squaregraph can equivalently be
described as the dual of a finite triangle-free hyperbolic line
arrangement. This suggests that the planar dual of an arbitrary locally finite triangle-free arrangement of
straight lines in the hyperbolic plane could be regarded as a (not necessarily finite) squaregraph. A pertinent
example of an infinite squaregrpah is the $\{ 4,5\}$ tesselation of the hyperbolic plane, which
yields a regular infinite squaregraph where every vertex has degree
5; see Figure 10 of \cite{Epp_orbi} (reproduced here as Figure~\ref{fig:45tess}),
shown in the Poincar\'e model of the hyperbolic plane. The hyperbolic lines of the arrangement  pass through opposite midpoints of the faces of each edge of this tiling.

By ``locally finite'' we mean that any open disk in the hyperbolic plane intersects only a finite number of lines of the arrangement.
Since closed disks are compact, this definition is equivalent to requiring that for every point $v$ there exists some $\epsilon>0$ such that the
$\epsilon$-disk centered at $v$ meets only finitely many lines. Since the hyperbolic plane is covered by a countable (directed) union of disks,
it follows that a locally finite arrangement comprises at most 	countably many lines. The result that the
dual of such an arrangement is a partial cube would then follow directly from known results on locally finite pseudoline arrangements \cite{BjLVStWhZi,EppFaOv}.
The lines of the arrangement correspond to the convex splits of the associated squaregraph, just as in the finite case. Note that,
by compactness of line segments, any two points of the hyperbolic plane are separated by no more that a finite number of lines from the locally finite 	arrangement.
The arrangement  is called {\it triangle-free} if no three lines intersect pairwise. If ${\mathcal A}$ is a hyperbolic line arrangement, then the complement of the union
of all lines from ${\mathcal A}$ in the hyperbolic plane is partitioned into cells via the intersection of the equivalence relations induced by the complementary open halfplanes
obtained from ${\mathcal A}.$ The {\it cell copair hypergraph} of ${\mathcal A}$ has the cells of ${\mathcal A}$ as its vertices and the sets of all cells included in open
halfplanes associated with the lines of ${\mathcal A}$ as its hyperedges. The {\it dual graph} of ${\mathcal A}$ has the cells as its vertices with two cells being adjacent
exactly when there exists no more than one line from ${\mathcal A}$ separating them.

\begin{theorem}
\label{theorem:infinite}
For a connected graph $G$ with finite vertex degrees, the following statements are equivalent:
\begin{itemize}
\item[(i)] $G$ has a locally finite squaregraph embedding, that is, $G$ can be drawn in the plane with its vertices
as points and its edges as disjoint rectifiable curves in such a way that every
bounded subset of the plane intersects finitely many vertices and
edges of $G$, every vertex either belongs to the boundary of an unbounded connected component of the complement of the drawing or has at
least four incident edges, and every bounded connected component of the complement of the drawing is the region bounded by a cycle of four edges of $G$;
\item[(ii)] for any vertex $v$ in $G$ and any integer $r$, the ball $B_r(v)$ induced by the vertices with distance at most $r$ from $v$ is isomorphic to a finite squaregraph;
\item[(iii)] $G$ can be covered by a countable chain of finite subgraphs $G_0\subset G_1\subset\cdots$ each of which is isomorphic to a squaregraph;
\item[(iv)] $G$ is a median graph that is isomorphic to a connected component of the dual graph of  some triangle-free  arrangement $\mathcal A$ in the hyperbolic
plane such that any cell represented by a vertex of $G$ is bordered by finitely many lines of $\mathcal A$. If, in addition, $\mathcal A$ is locally finite, then $G$ is the dual graph of $\mathcal A$.
\end{itemize}
\end{theorem}

We regard Theorem \ref{theorem:infinite} as the natural extension of the concept of squaregraph to the infinite case, as a number of intuitive concepts carry over from the finite to the infinite. This does not preclude that some countable median graphs beyond infinite squaregraphs can also be represented in the hyperbolic plane as dual graphs of line arrangements. For instance, it is not difficult to see that every countable tree is the dual of a locally finite and triangle-free line arrangement in the hyperbolic plane.

From Theorem~\ref{theorem:infinite} we conclude that any finite squaregraph is the dual of a finite triangle-free arrangement of lines in the hyperbolic plane. However, the
arrangements whose dual graphs host infinite locally finite squaregraphs as connected components are not necessarily locally finite, the simplest example illustrating this ``anomaly" being the square grid $\mathbb{Z}^2.$  Figure~\ref{fig:twoarrangements} presents two  such non locally finite arrangements for  $\mathbb{Z}^2.$ The dual graph of the first arrangement has five
different connected components, one of which is the square grid and
the other four are infinite paths. The dual graph of the second arrangement has the sixth connected component composed of a single cell.

\begin{figure}[h]
\begin{tabular}{c @{\qquad} c}
\includegraphics[width=0.35\textwidth]{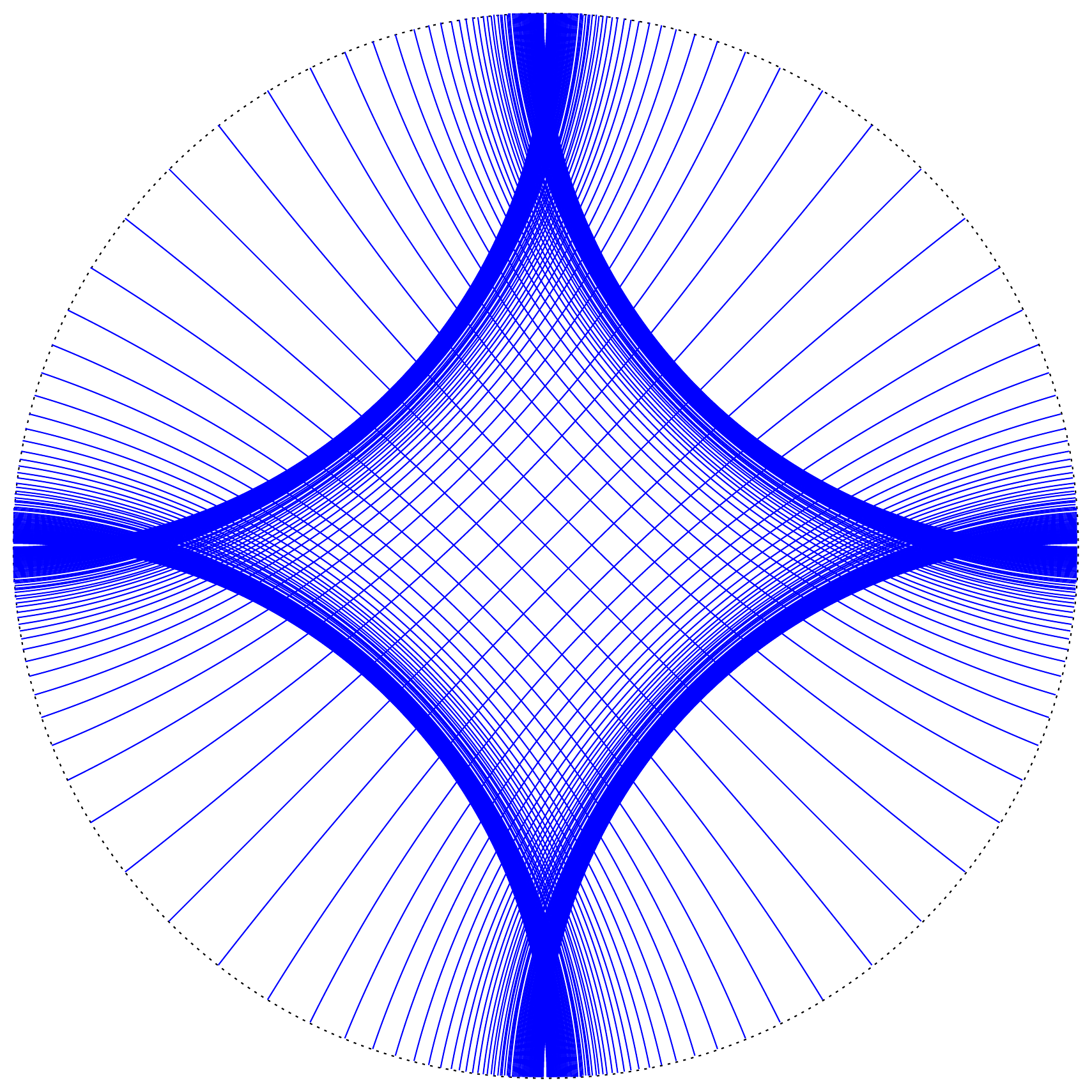}&
\includegraphics[width=0.35\textwidth]{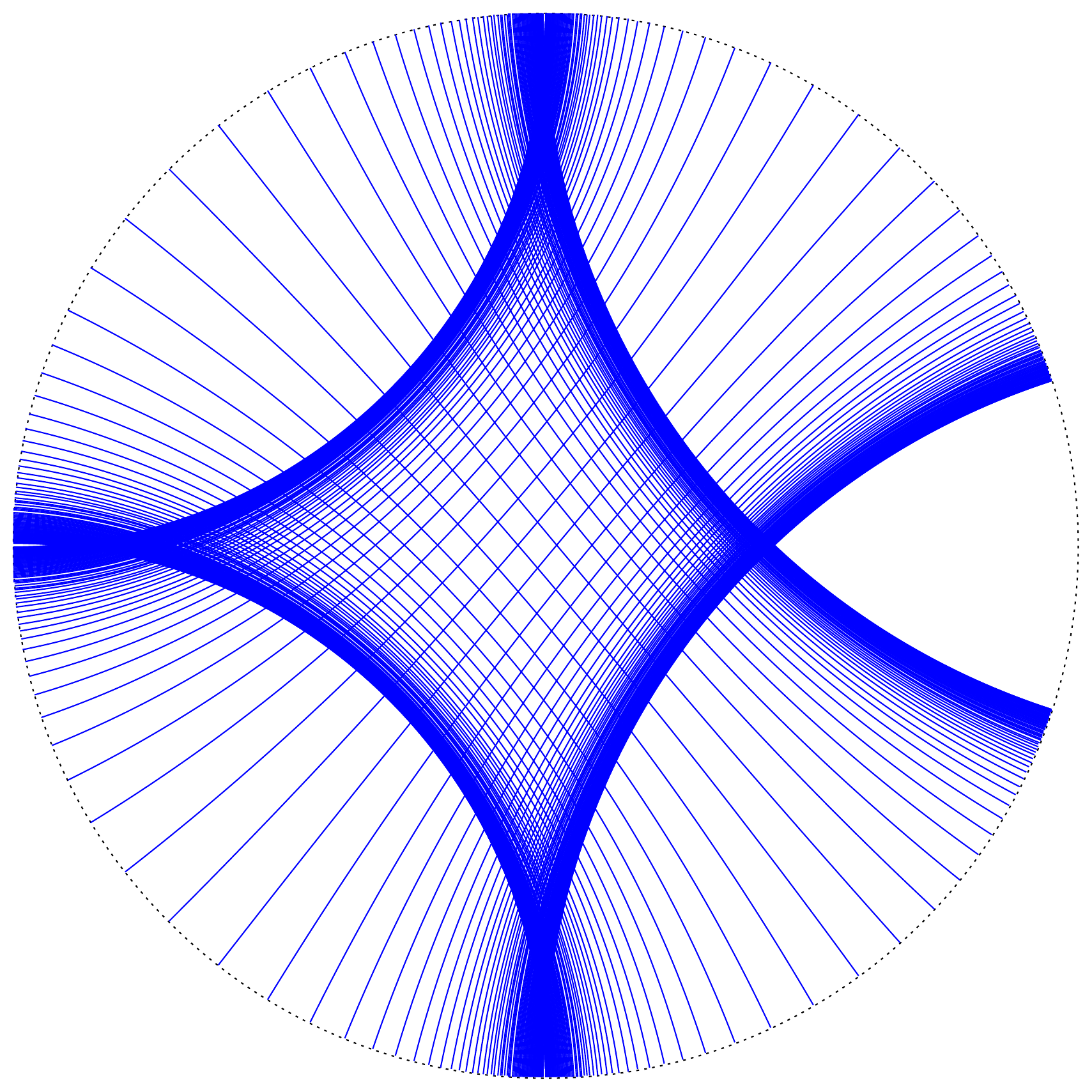}\\
(a)&(b)\\
\end{tabular}
\caption{Two different arrangement representations of the square grid.}
\label{fig:twoarrangements}
\end{figure}

\begin{figure}[h]
\begin{tabular}{c @{\qquad} c}
\includegraphics[width=2in]{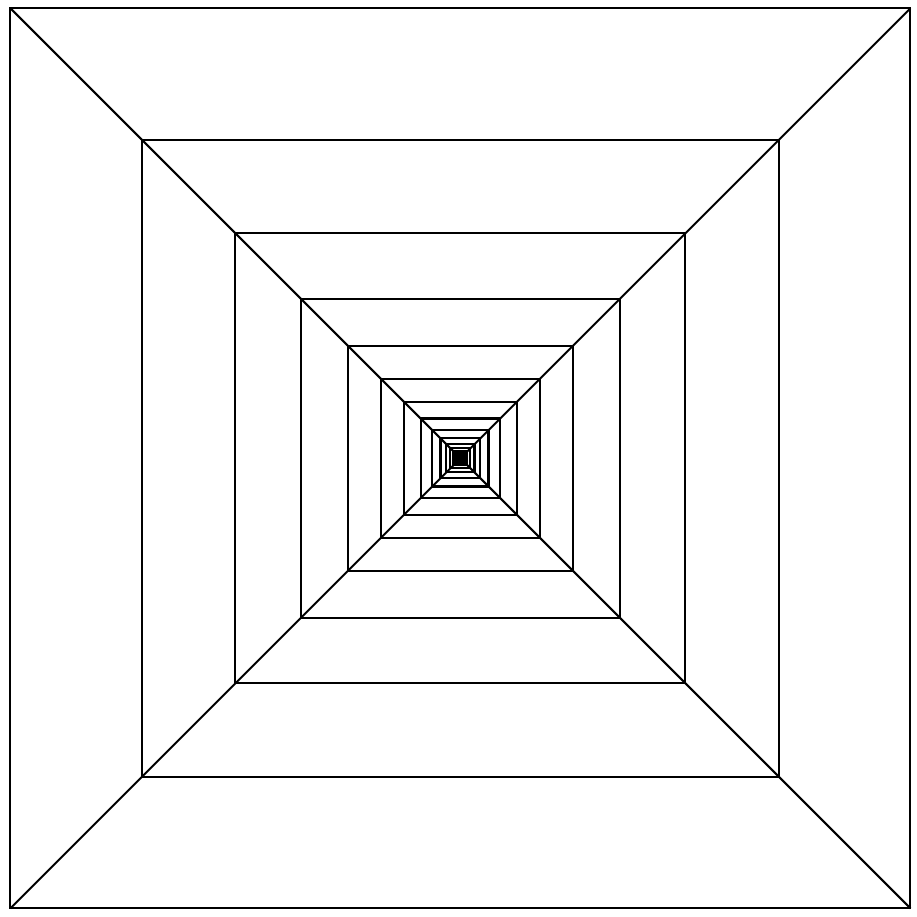}&
\includegraphics[width=2in]{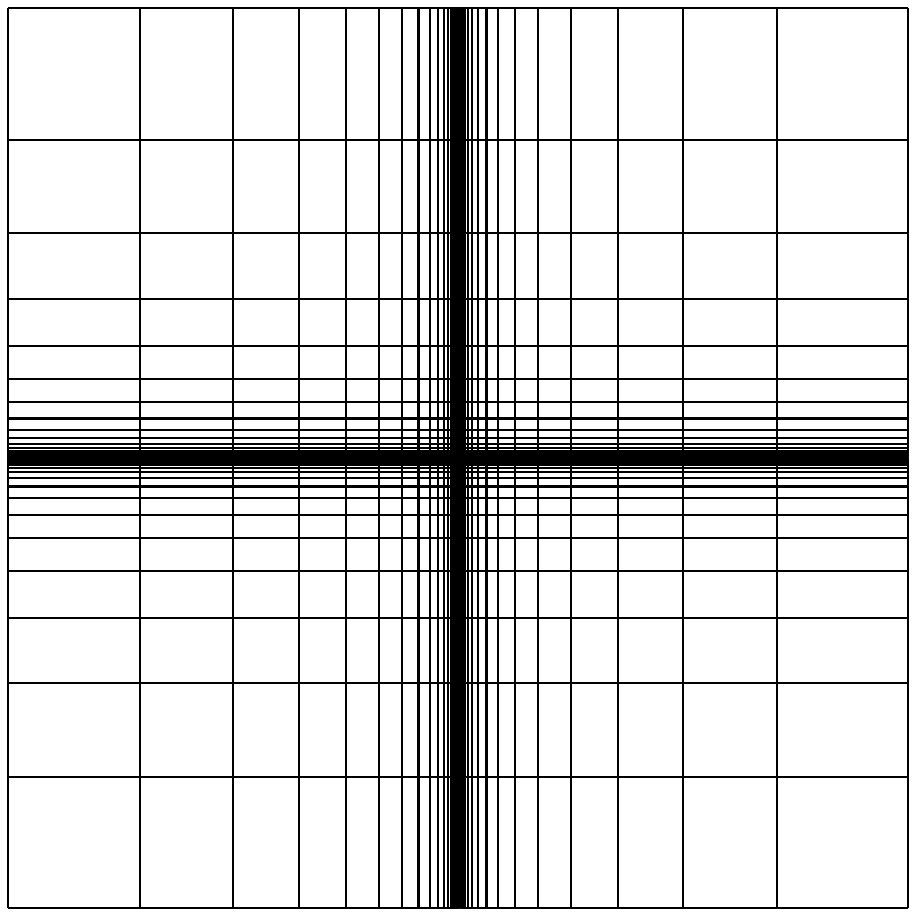}\\
(a)&(b)\\
\end{tabular}
\caption{Two infinite planar graphs, with quadrilateral faces and no inner vertex of degree less than four, that should not be considered to be infinite squaregraphs.}
\label{fig:squaretunnel}
\end{figure}

The example in Figure~\ref{fig:squaretunnel}a should not be considered a squaregraph, although it
meets many of the defining criteria: it is a plane median graph with finite vertex degrees, every
open connected component of its complement is either unbounded or a quadrilateral, and every
vertex either lies on the boundary or has degree at least four. However, it fails the local finiteness
requirement of Theorem \ref{theorem:infinite}(i), the finite induced subgraphs $B_r(v)$ are not squaregraphs for any $r\ge 2,$
and the family of curves it is dual to do not form a line arrangement (they consist of two crossing
lines and a collection of concentric closed curves surrounding the crossing point). The example in
Figure~\ref{fig:squaretunnel}b again superficially resembles a squaregraph, and is the dual of a
triangle-free line arrangement, but the line arrangement is not locally finite and the graph is not connected.

%The example in Figure~\ref{fig:squaretunnel}a, however, should not be considered a squaregraph, although it meets many of the defining criteria: it is a connected planar median graph with finite vertex degrees, every open connected component of its complement is either unbounded or a quadrilateral, and every vertex either lies on the boundary or has degree at least four. However, it fails the local finiteness requirement of Theorem~\ref{theorem:infinite}(i), the finite induced subgraphs $B_r(v)$ are not squaregraphs for any $r\ge 2$, and the family of curves it is dual to do not form a line arrangement (they consist of two crossing lines and a collection of concentric closed curves surrounding the crossing point). The example of Figure~\ref{fig:squaretunnel}b again superficially resembles a squaregraph, and is the dual of a triangle-free line arrangement, but the line arrangement is not locally finite and the graph is not connected.

It is tempting to attempt to generalize the description of squaregraphs as duals of hyperbolic line
arrangements, and consider instead pseudoline arrangements. However, the dual of any triangle-free pseudoline arrangement is isomorphic to the dual of a hyperbolic
line arrangement having the same ends of lines in the same cyclic order. Therefore, pseudolines give no added generality.

\section{Embedding into products of trees}
\label{sec:tree-products}

One of the important issues in metric graph theory, and more
generally, in metric geometry, is a succinct representation of
graphs (or metric spaces), which can be of use in the design of
efficient algorithms for solving various metric problems on the
input graph or space. For instance, one may search for an exact or
approximate isometric embedding of graphs or spaces into
low-dimensional Euclidean or $l_1$-spaces or into the Cartesian
product of a small number of trees; labelings of the vertices of the
graph by the coordinates of such an embedding can be used to form
data structures for efficient navigation in these graphs
\cite{ChDrVa_CAT}. However, as is shown in \cite{BaVdV}, for every
$k\ge 3$ the problem of recognizing the graphs that are
isometrically embeddable into the Cartesian product of $k$ trees is
NP-complete, even when restricted to cube-free median graphs. In
fact, a median graph $G$ is isometrically embeddable into the
Cartesian product of $k$ trees if and only if the incompatibility
graph of convex splits is $k$-colorable \cite{BaVdV}. Moreover,
every connected graph is the incompatibility graph of some median
graph. In the case of cube-free median graphs, the corresponding
incompatibility graph is triangle-free, but this property alone is
not enough to ensure a low chromatic number \cite{Nill}. On the
other hand, it is shown in \cite{Epp_lattice} how to embed a partial
cube optimally into the Cartesian product of the least number of
paths. In a companion paper \cite{BaChEpp}, it will be proven that
the graphs isometrically embeddable into the Cartesian product of
two trees are exactly the cube-free median graphs not containing
induced odd cogwheels. In particular, this shows that the
squaregraphs in which all inner vertices have even degree (even
squaregraphs) can be embedded into the Cartesian product of two
trees. For other classes of plane graphs one has similar results,
for example, ``benzenoid systems" (i.e., plane graphs with all inner
faces of length six and all inner vertices of degree three) can be
isometrically embedded into the Cartesian product of three trees
\cite{Ch_benzen}. This result directly led to a linear time
algorithm for computing the so-called Wiener number of benzenoids
\cite{ChKl}.

We will employ the facts that finite squaregraphs $G$ are plane median
graphs such that the convex splits of $G$ induce a 2-compatible
circular split system on the boundary cycle of $G$. This establishes
a bijection between the embeddings of $G$ into the product of $k$
trees and the colorings of the underlying triangle-free circle graph
of those circular splits with $k$ colors. Since it is known that
triangle-free circle graphs are 5-colorable but not necessarily
4-colorable~\cite{Ag,JeTo}, we obtain the main result of this paper:

\begin{figure}[t]
\centering\includegraphics[height=3.7in]{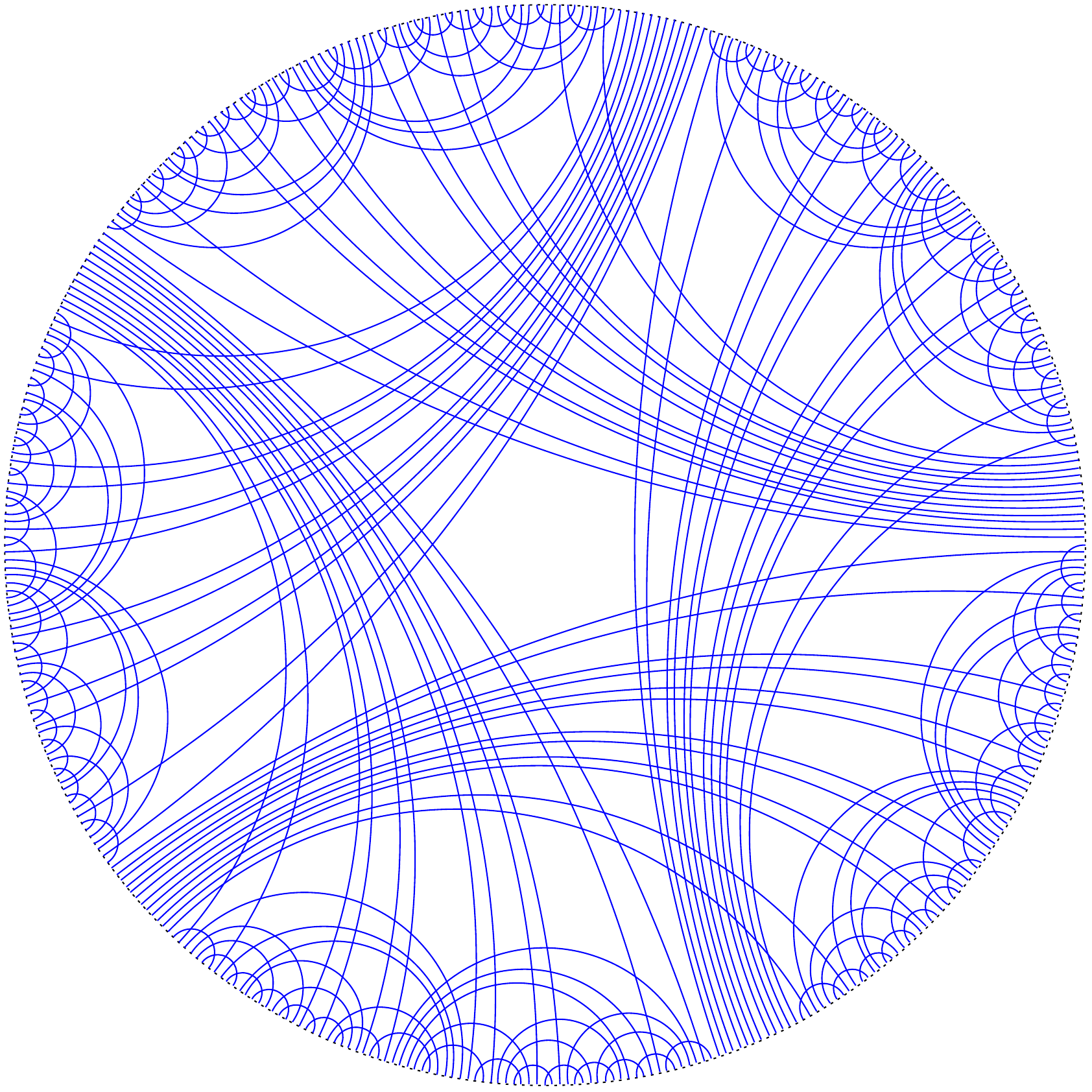} \caption{Ageev's
5-chromatic  graph of 220 chords (drawn hyperbolically).}
\label{fig:220chords}
\end{figure}

\begin{figure}[t]
\centering\includegraphics[height=3.7in]{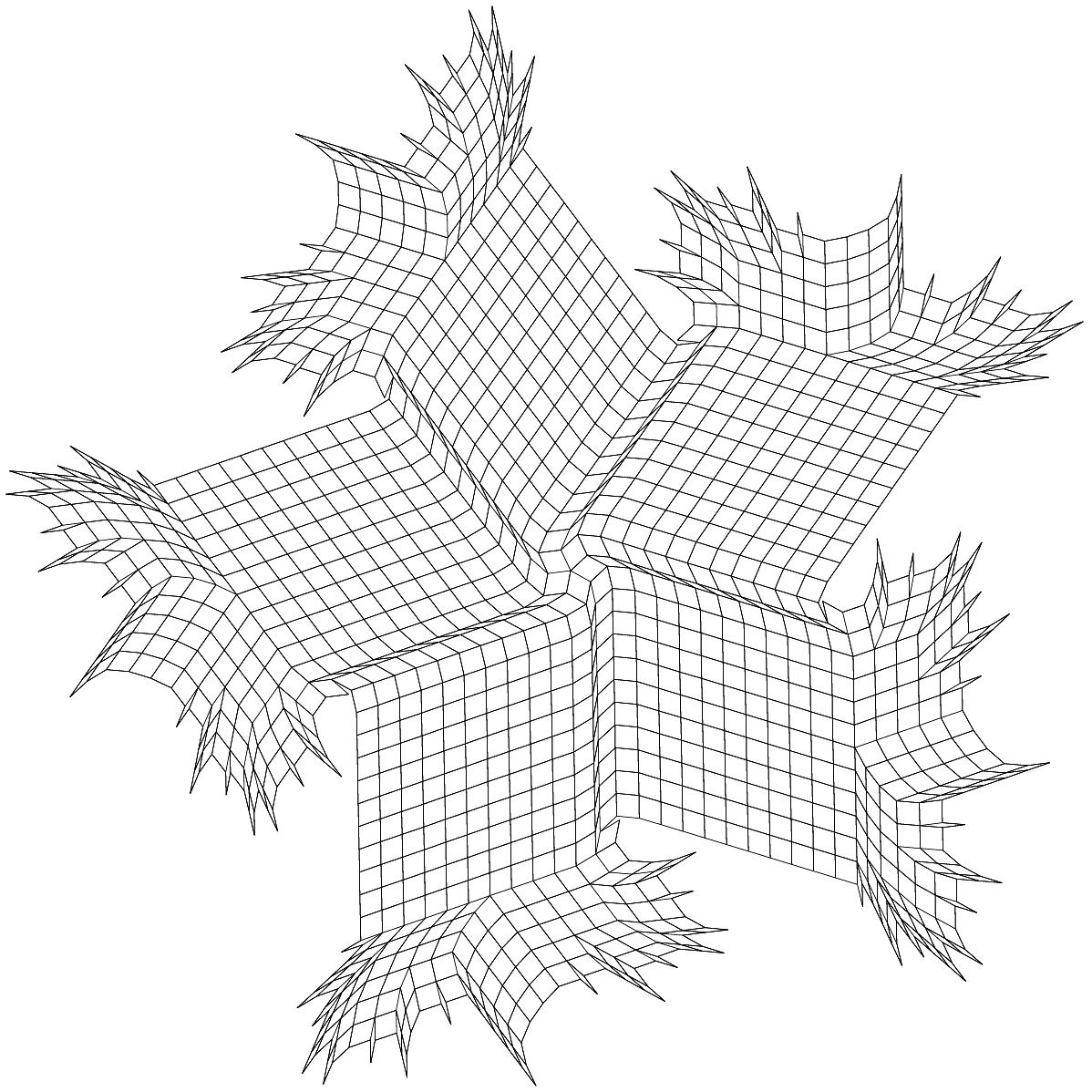} \caption{Squaregraph
dual to Ageev's graph. For the algorithms used to construct this drawing, see~\cite{EppWor-08}.}
\label{fig:dual220chords}
\end{figure}

\begin{theorem}
\label{theorem:five_trees}
Every finite squaregraph $G$ can be isometrically embedded
into the Cartesian product of at most five trees. There exist
2-connected squaregraphs with maximum degree 5 that cannot be
embedded into the Cartesian product of just four trees. A
squaregraph not containing any induced (or isometric) $2\times 2$
grid (i.e., $K_{1,2}\Box K_{1,2}$) can be isometrically embedded
into the Cartesian product of at most three trees. $G$ can be isometrically
embedded into the Cartesian product of at most two  trees exactly when $G$ is even,
that is, every inner vertex has even degree.
\end{theorem}

The preceding result result immediately extends to the infinite case, in particular to 	
infinite squaregraphs:

\begin{theorem}
\label{theorem:five_trees_infinite}
Every median graph $G$ not containing any induced (or isometric)
cube, $K_2\Box K_{1,3}$, or suspended cogwheel can be isometrically
embedded into the Cartesian product of at most five trees. If, in addition, the
$2\times 2$ grid is a forbidden induced (or isometric) subgraph, then an
isometric embedding into the Cartesian product of at most three trees is 	
guaranteed. Finally, $G$ can be isometrically embedded into the Cartesian product 	
of at most two trees exactly when $G$ does not contain an induced odd 	cogwheel.
\end{theorem}

The 2-connected squaregraphs not containing $2\times 2$ grids are
negatively curved. For a 2-connected plane graph $G$ the notion of
curvature is gleaned from the Gauss-Bonnet formula for piecewise
Euclidean 2-complexes. Specifically, the {\it curvature} at a vertex
$x$ with degree $deg(x)$ in $G$ equals
$$1-deg(x)/2+\sum_{F}1/|F|, \mbox{\hspace*{0.3cm} if } x \mbox{ is
an inner vertex of } G,$$ and
$$1/2-deg(x)/2+\sum_{F} 1/|F|, \mbox{ \hspace*{0.3cm} if } x \mbox{ is on the boundary
cycle of } G,$$ where the sums extend over all faces  $F$  incident
with  $x.$

\medskip\noindent
Note that Ishida's ``combinatorial'' curvature, often employed for
polyhedral graphs, is a dual notion. The key issue of graph
curvature is whether the curvature of all vertices is always of the
same type (either negative, or nonpositive, or zero, or
nonnegative, or positive, respectively). The curvature of every
finite 2-connected squaregraph is nonpositive, and this also extends
to the infinite case under the assumption of finite vertex degrees.
Specifically, the curvature of a vertex $v$ on the boundary cycle
$C$ equals  $1/4-deg(v)/4+1/|C|,$ which attains $0$ only for the
$4$-cycle, and of an inner vertex $u$ equals  $1-deg(u)/4,$ which
attains $0$ exactly when $deg(u)=4.$ Therefore a 2-connected
squaregraph $G$ has negative curvature exactly when it is not the
4-cycle and has no 4-cogwheel (that is, the $2\times 2$ grid).

\section{Proofs for Section~\ref{sec:split-to-median}}
\label{proofs:split-to-median}

\bigskip\noindent
{\bf Proof of Proposition~\ref{prop:Helly*}:} If $\mathcal H$ is $^*$Helly, then the collection of the hyperedges from
$\mathcal H$ containing at least two of the points $u,v,w$ is anchored and
pairwise intersecting, whence its intersection is nonempty. Conversely, assume
that the triplet condition is satisfied. Let $\mathcal F$ be any anchored and
pairwise intersecting collection of hyperedges. For any $u\in X,$ the collection of
hyperedges from $\mathcal F$ containing $u$ misses only finitely many hyperedges from
$\mathcal F$. Since the finite set ${\mathcal F}'$ of these hyperedges satisfies
the triplet condition (as does any subfamily of $\mathcal F$), the intersection of ${\mathcal F}'$ is nonempty by the
second part of the proof of \cite[Theorem 10 of Ch. 1]{Be}, which uses induction
on $\#{\mathcal F}'.$ Choose any $v$ from $\bigcap {\mathcal F}'.$ Then, by the
same argument, the finite set ${\mathcal F}''$ of hyperedges from $\mathcal F$
separating $u$ and $v$ intersects in at least one point $w.$ By the choice of $u,v,w$ every hyperedge of
$\mathcal F$ contains at least two of $u,v,w,$ and therefore $\bigcap{\mathcal F}$
is nonempty.

\bigskip\noindent
{\bf Proof of Proposition~\ref{prop:Mulder_Schrijver}:} (i)$\Rightarrow$(ii): If $\mathcal H$ did not separate two distinct points $x$
and $y,$ then we could add the complementary pair $\{ x\}, V\setminus \{ x\}$ to $\mathcal E$ without violating the $^*$Helly property.

(ii)$\Rightarrow$(iii): As in the finite case \cite{BaHe} define the segment
$$u\circ v=v\circ u=\cap \{ E\in {\mathcal E} \mid u,v\in E\}$$
for each pair $u,v$ of points. Then $u\circ v\cap v\circ w\cap w\circ u\ne \emptyset$ by
Proposition~\ref{prop:Helly*}. If there were two different points, $x$ and $y,$ in this
intersection, then we could find some $F\in {\mathcal E}$ separating $x$ and $y.$
Necessarily, $F$ would include two of $u,v,w$ and hence a segment
containing  both $x$ and $y,$ a contradiction. Therefore the intersection is a
singleton. In particular, $u\circ u=\{ u\}.$ Monotonicity, that is,
$w\in u\circ v \Rightarrow u\circ w\subseteq u\circ v,$ is trivial by definition
of the segments. We conclude that the segments meet Sholander's conditions for
a median algebra; see \cite[Theorem 2.1]{BaHe}. If some segment contained an infinite
chain $u_1,u_2,\ldots$ such that
$u\circ v \supseteq u_1\circ v\supseteq u_2\circ v\supseteq\ldots,$ then infinitely
many members of ${\mathcal E}$ would separate $u$ and $v$, contradicting the assumption that $\mathcal H$ is discrete.  Hence the median algebra
is itself discrete and constitutes a median graph $G$ on $V$ where the intervals coincide
with the corresponding segments. It remains to show that $\mathcal H$ comprises all
halfspaces of $G.$ By definition of
the segments, every member of $\mathcal E$ is a halfspace of $G.$ Conversely, for
each halfspace $F$ of $G$ there exists an edge $uv$ of $G$ such that $F$ is the unique
halfspace separating $u$ and $v,$ which must belong to $\mathcal H$ because $\mathcal H$
separates the points.

(iii)$\Rightarrow$(i):  By Corollary \ref{Helly_convex}, $\mathcal H$ has the $^*$Helly property.
To show that $\mathcal H$ is maximal, suppose by way of contradiction that we could extend
$\mathcal H$ to a larger $^*$Helly copair hypergraph by adjoining two complementary sets $A$ and
$B$ where $A$ is not convex in $G$.  Because $A$ is non-convex, there exist $u,v$ in $A$ such that a shortest path from $u$ to $v$ passes through some neighbor $x$ of $u$ belonging to $B$.
Let ${\mathcal D}_u$ and ${\mathcal D}_v$ be the sets of halfspaces
containing the edge $ux$ and the interval $I(v,x),$ respectively.
Then the collection ${\mathcal D}:={\mathcal D}_u\cup {\mathcal D}_v\cup \{ A\}$ intersects in pairs and is anchored. However,
$$\bigcap{\mathcal D}=\bigcap {\mathcal D}_u\cap \bigcap {\mathcal D}_v\cap A=\{ u,x\}\cap I(v,x)\cap A=\emptyset,$$
contradicting the $^*$Helly property.

\bigskip\noindent
{\bf Proof of Proposition~\ref{prop:Nica}:} For $\mathcal D$ as described in the proposition, one can find a finite set $Z$
intersecting every member of $\mathcal D$.  Pick any $z\in Z$. For each pair $A, B \in {\mathcal H}\setminus{\mathcal D}$
of complementary sets, choose the one that intersects each member of the pairwise intersecting system
${\mathcal D}\cup \{ Z\}$ and -- in case that both $A$ and $B$ meet
this criterion -- that additionally contains $z$. The resulting set family $\mathcal F$ is clearly anchored, and as we now show it is pairwise intersecting. For, the added sets by construction have nonempty intersections with each member of $\mathcal D$. And, if two sets $A'$ and $A''$ from ${\mathcal F}\setminus{\mathcal D}$  were disjoint,
then one of the sets, say $A'$, would not contain $z$. But we can only include such a set $A'$ in $\mathcal F$ when its complementary set $X\setminus A'$ is disjoint from some member of $\mathcal D$, and if this were the case then $A''\subset X\setminus A'$ would also be disjoint from that member, contradicting the fact that each member of $\mathcal F$ intersects each member of $\mathcal D$. This contradiction shows that $\mathcal F$ is pairwise intersecting, and it is clearly maximal with this property as it includes one member from each complementary pair in $\mathcal H$. This proves the first assertion.

Note that the augmentation of $\mathcal H$ to $[{\mathcal H}]$ does not fill empty pairwise intersections, which entails that compatibility of the corresponding splits is preserved.
By definition, $[{\mathcal H}]$ is a copair hypergraph on the extended set $V=[X]$ such that every member of $[{\mathcal E}]$ restricts to a
member of $\mathcal E$ on $X.$ Then the trace on $X$ of any anchored pairwise intersecting subset of $[{\mathcal E}]$ is an anchored pairwise
intersecting subset of $\mathcal E$, which has a nonempty intersection in $[X]$ by construction.
If $v_{\mathcal F}$ and $w = v_{{\mathcal F}'}$ are new vertices associated with two maximal anchored pairwise intersecting collections ${\mathcal F}$ and ${\mathcal F}'$, then there are finite sets $Z$ and $Z'$ met by all members of ${\mathcal F}$ and ${\mathcal F}'$, respectively. Hence the members of $[E]$ that separate $v_{\mathcal F}$ from $w$ all intersect the finite set $Z \cup Z'$ and hence must be finite in number because $H$ is discrete. If instead $w$ is a pre-existing vertex, then by a similar argument (letting $\{w\}$ play the role of $Z'$) there are again only finitely many members of $[E]$ that separate $v_{\mathcal F}$ from $w$.
Therefore $[{\mathcal H}]$ is a discrete
$^*$Helly copair hypergraph, which obviously separates the points.
By Proposition~\ref{prop:Mulder_Schrijver}, $[{\mathcal H}]$ is the halfspace
hypergraph of a median graph $G$ on $V.$ Since every new vertex $v\in V\setminus X$ is the unique vertex $v=v_{\mathcal F}$ of $[{\mathcal H}]$ that
fills some empty intersection of a maximal anchored pairwise intersecting collection ${\mathcal F}\subseteq {\mathcal E},$ no trace of
$[{\mathcal E}]$ on any proper subset $W$ of $V$ with $X\subseteq W$ could yield a $^*$Helly copair hypergraph. Hence the smallest median
subgraph of $G$ with vertex set extending $X$ must be all of $G.$

\section{Proofs for Section~\ref{sec:circular-split}}
\label{proofs::circular-split}

We start by showing that squaregraphs are median graphs. This
follows from a more general bijection between median graphs and
1-skeletons of cubical complexes with global nonnegative curvature
\cite{BaCh_survey,Ch_CAT,Ro}. Here we will give a self-contained
proof of this property, which will make use of an elementary
counting argument.

\begin{lemma}\label{4corners}\cite{SoZaPr}
Every finite 2-connected squaregraph $G$ contains at least four vertices of degree 2.  Consequently, a finite squaregraph $G$
which is not a tree has at least four vertices of degree at most 2. The plane subgraph $G'$ of a finite squaregraph $G$ which is induced by all vertices of $G$ lying either on some simple cycle $C'$ or inside the plane region bounded by $C'$ is a 2-connected squaregraph. Consequently, any articulation point is on the boundary and any 4-cycle of a squaregraph $G$ is an inner face.
\end{lemma}

\begin{proof}
Let $f$ denote the number of inner faces of $G,$ $m$ the number of
edges, $n$ the number of vertices, $b$ the number of vertices
incident with the outer face, and $n_2$ the number of vertices of
degree 2. Then $f-m+n=1$ and  $4f+b=2m$ hold according to Euler's
formula and the hypothesis that all faces are quadrangles.
Eliminating $f$ yields $4n-2m-b=4.$ The information on the vertex
degrees is turned into the inequality $2m\ge
4(n-b)+3(b-n_2)+2n_2=4n-b-n_2,$ whence $n_2\ge 4n-b-2m=4,$ as
required.

If $G$ is neither 2-connected nor a tree, then it contains a 2-connected block $B,$ which has four vertices of
degree 2 in $B.$ If any such vertex $u$ is an articulation point of $G,$ then select an end block $B'$ of $G$ that has $u$ as its gate in $B.$
Necessarily, $B'$ has at least one vertex $u'$ of degree 1 or 2 that is not an articulation point. In this way, we obtain four distinct vertices of degrees at most two in $G.$

As for the third assertion of the lemma, the plane graph $G'$ which is enclosed by the chosen cycle $C'$ in the plane is included in some 2-connected block $B$ of $G$ and has $C'$ as
its boundary. Every inner vertex of $G'$ is also an inner vertex of $G$ and hence has degree at least four, and every inner face of $G'$ is an inner face of $G$ and thus a 4-cycle, whence $G'$ is a squaregraph. By virtue of the first assertion of the lemma, $C'$ contains a degree-two vertex $u$ of $G'$ along with its two neighbors $t$ and $w$. If the fourth vertex $x$ of the inner face $F$ to which $t$, $u$, and $w$ belong lies on $C'$, then $C'$ is the modulo 2 sum of $F$ and at most two simple cycles. Since any articulation point of $G'$ would be an inner vertex of $G'$, a straightforward inductive argument shows that $G'$ must be 2-connected. Finally, to prove the fourth assertion, if the given squaregraph $G$ contains some 4-cycle $C'$ that is not an inner face of $G$, then the squaregraph $G'$ enclosed by $C'$ in the plane will have fewer than four vertices of degree two, contradicting the first assertion.
\end{proof}

%Note that the assumption of 2-connectivity in the lemma is necessary: there exist arbitrarily large finite squaregraphs (in fact, trees) with no degree two vertex.

\begin{lemma}\label{median}
Every finite squaregraph $G$ is a cube-free and
$K_2\Box K_{1,3}$-free median graph.
\end{lemma}

\begin{proof}
To establish that $G$ is median, suppose by
way of contradiction that $G$ contains a triplet $x,y,z$ such that
$I(x,y)\cap I(y,x)\cap I(x,z)=\emptyset.$ We can suppose without
loss of generality that each pair of intervals intersect only at
their common bounding vertex, for otherwise, if say $I(x,y)\cap I(x,z)\ne
\{x\},$ we could replace $x$ by a vertex from this intersection
while still retaining a contradiction. This implies that $x$, $y$,
and $z$ belong to a common 2-connected component of $G$, whence we
may assume that $G$ is already 2-connected. Now pick three shortest
paths $J(x,y)$, $J(y,z)$, and $J(z,x)$, such that the subgraph $G'$
defined by the cycle $C'=J(x,y)\cup J(y,z)\cup J(z,x)$ contains a
minimum number of inner faces. Since $G'$ is a
squaregraph by the second assertion of Lemma \ref{4corners},  the
boundary cycle $C'$ of $G'$ contains at least four vertices of
degree 2 (in $G'$), by the first assertion of the same lemma. Let
$u$ be such a vertex different from $x,y,z,$ say $u\in J(x,y).$
Denote by $u'$ the vertex opposite to $u$ in the unique inner face
of $G'$ containing $u$. Then replacing in $J(x,y)$ the vertex $u$ by
$u'$ results in a new shortest path $J'(x,y)$ between $x$ and $y$,
and the cycle $J'(x,y)\cup J(y,z)\cup J(z,x)$ defines a squaregraph
$G''$ having fewer inner faces than $G'.$ This contradiction
establishes that $I(x,y)\cap I(y,x)\cap I(x,z)\ne\emptyset$ for all
$x,y,z\in V.$ Now, if $G$ is not median, then it must contain an
induced $K_{2,3}$ \cite{Mu}. However, any plane embedding of $K_{2,3},K_2\Box K_{1,3},$
or the cube has a 4-cycle that is not an inner face, so from Lemma~\ref{4corners}
we infer that $G$ is $K_2\Box K_{1,3}$-free and cube-free.
\end{proof}

The Djokovi\'c relation $\Theta$ \cite{Dj} on a bipartite graph
$G=(V,E)$ is defined on the edges of $G$ by
$$uv \Theta xy \Leftrightarrow \mbox{ either } x\in W(u,v) \mbox{
and } y\in W(v,u), \mbox{ or } y\in W(u,v) \mbox{ and } x\in
W(v,u)$$ for $uv, xy \in E.$ The relation $\Theta$ is always symmetric and idempotent; it is transitive if
and only if the splits $\sigma(uv)$ are convex for all $uv \in E,$
which is equivalent to isometric embeddability of $G$ into some
hypercube \cite{Dj}. That is, $G$ is a partial cube if and only if $\Theta$
is an equivalence relation. On a median graph the Djokovi\'c relation
$\Theta$ coincides with the transitive closure $\Psi^*$ of the
relation $\Psi$ defined by
$$uv\Psi xy \Leftrightarrow \mbox{ either } uv = xy \mbox{ or } uv \mbox{ and } xy \mbox{ are opposite edges of
some 4-cycle,}$$ according to Lemma 1 of \cite{BaCh_cellular}. For
squaregraphs, ``4-cycle" can be replaced by ``inner face" in the
latter definition. For an edge $uv$ of a squaregraph $G,$ we denote
by $\Theta(uv)$ the equivalence class of $\Theta$ containing $uv.$
By $Z(uv)$ we denote the subgraph induced by the union of
$\Theta(uv)$ with all inner faces of $G$ sharing common edges with
$\Theta(uv)$  and call $Z(uv)$ the \emph{zone} of $\Theta(uv).$ Set
$P(u,v)=Z(uv)\cap W(u,v)$ and $P(v,u)=Z(uv)\cap W(v,u).$ Since $G$
is median, $P(u,v)$ and $P(v,u)$ are convex and therefore gated sets
\cite{Mu}. The convex sets $P(u,v)$ and $P(v,u)$ are isomorphic via
the matching induced by the edges from the cutset $\Theta(uv),$ and
the zone $Z(uv)$ is isomorphic to $K_2\Box P(u,v)$ \cite{Mu}.
Therefore from Lemma \ref{median} we immediately obtain the
following observation.

\begin{lemma} \label{strip} \cite{SoZaPr} For every edge $uv$ of a
finite squaregraph $G,$ the
zone $Z(uv)$ is a ladder comprising the edges from the cutset
$\Theta(uv)$ and the convex paths $P(u,v)$ and $P(v,u).$
\end{lemma}

\bigskip\noindent
{\bf Proof of Proposition~\ref{prop:unique_cyclic_order}:}
From Lemma~\ref{strip}, we know that the convex splits of $G$ restrict
to circular splits on $C$ when $G$ is 2-connected. To establish this implication for all
squaregraphs, we proceed by induction on the number of 2-connected
components of $G.$ Suppose that $G$ is obtained as an amalgamation
of two squaregraphs $G'$ and $G''$ along a common boundary vertex
$x.$ By induction assumption, the graphs $G'$ and $G''$ admit
circular representations $C'$ and $C''$. Let $x'$ and $x''$ be the
copies of $x$ in $C'$ and $C''$, respectively. Let $y'$ be a
neighbor of $x'$ in $C'.$  To obtain a circular representation $C$
of $G,$ we identify the two copies of $x$ in $C'$ and $C'',$ take
the circular representation $C'$ and place the remaining vertices of
$C''$ next to $x'$ but before $y'$ following the circular order of
$C''.$  To see that the traces of convex splits of $G$ on $C$ indeed
form a circular split system notice that any convex split
$\sigma(uv)$ of $G$ can be derived from the convex splits of $G'$
and $G''$ in the following way. Suppose without loss of generality
that $uv$ belongs to $G'$ and denote by $W'(u,v)$ and $W'(v,u)$ the
halfspaces defined by the edge $uv$ in $G'.$ If $x\in W'(v,u),$ then
$W(u,v)=W'(u,v)$ and $W(v,u)=W'(v,u)\cup V''$ where $V''$ denotes the vertex set of $G''$. If $x\in W'(u,v),$
then $W(u,v)=W'(u,v)\cup V''$ and $W(v,u)=W'(v,u).$ From the
construction of $C$ we infer that  the traces of $W(u,v)$ and
$W(v,u)$ on $C$ define indeed a circular split, establishing the claim in the proposition.

In the argument above, reversing $C'$ before amalgamating it with $C''$, or choosing the neighbor $y'$ of $x'$ differently, produces another cyclic order except in the cases enumerated in the proposition. It remains to prove uniqueness of the cyclic order for 2-connected squaregraphs.

In a finite 2-connected squaregraph $G$, the outer face $C$ is a simple cycle of length $2k\ge 4$. Every pair of consecutive
boundary vertices are connected by an edge separated by some split. Thus, each of the $k$ convex splits of $G$ separates some two pairs of
boundary vertices that are not distinguished by any others of the splits. Suppose that $C^*$ is a cycle (not necessarily a subgraph of $G$)
having the same vertices as $C$ but different from $C$ such that the split system of $C$ also yields circular splits on $C^*$. Now, if this
system harbored two different splits separating the same edge of $C^*$, then by the pigeon hole principle at least one edge $uv$ of
$C^*$ would not get separated by a convex split of $G$, contradicting separation of distinct points in $G$. Hence every edge of
$C^*$ is cut by exactly one split. Take an edge $ab$ of $C^*$ that is not an edge of $C$, and let $\{ A,B\}$ be the unique split of the system
separating $a$ and $b$. Since $\{ A,B\}$ is also a circular split of $C^*$, it separates a second edge, $a'b'$, of $C^*$. We may assume that
$a,a' \in A$ and $b,b' \in B$. If $a$ and $b$ were not adjacent in $G$, then we would have at least two convex splits separating $a$ and $b$,
in conflict with the uniqueness property just shown. Therefore $ab$ and, for the same reason, $a'b'$ are edges of $G$. It follows that $C^*$ is
a subgraph of $G.$

Since $ab$ is a chord of $C,$ it is a cut edge of $G$ (by virtue of Lemma~\ref{4corners}) and thus $G$ is a (gated) amalgam \cite{BaCh_survey} of 2-connected squaregraphs $G_1$
and $G_2$ along $\{ a,b\}$. Then $a'b'$ is an edge of $G_1$,
say. The convex split $\{ A,B\}$ of $G$ separating $a$ and $b$ (as well as $a'$ and $b'$) separates exactly one edge $uv$ on $C$ that is on the outer
face of $G_2$ where $u \in A$ and $v \in B$. Let $xy$ be the edge (not necessarily distinct from $ab$) opposite to $uv$ on the 4-cycle of
$G_2$ containing $uv$. Then the convex split of $G$ separating $uv$ from $xy$ also separates $uv$ from $G_1$. Since $u$ is between $a$
and $a'$ and $v$ is between $b$ and $b'$ on $C^*$, this convex split cannot restrict to a circular split on $C^*$, a contradiction.
This establishes Proposition~\ref{prop:unique_cyclic_order}.

\medskip
The second part of the preceding proof can be shortcut by observing that the subgraph $H$ induced by the boundary cycle $C$ in the
squaregraph $G$ is outerplanar.
It is well known that a 2-connected outerplanar graph has a unique Hamiltonian cycle; see e.g. Lemmas 5 and 6 of \cite{LeSt}.
Then $C$ and $C^*$, both
constituting Hamiltonian cycles of $H$ by the first part of the proof, must coincide, contrary to the assumption.
Note that, conversely, it follows (by induction) from the second part of the above proof that every bipartite,
2-connected, finite outerplanar graph can occur as the subgraph $H$ induced by the boundary of a
2-connected finite squaregraph. Obviously, the subgraph $H$ does not in general determine the squaregraph uniquely.

\section{Proofs for Section~\ref{sec:smallgen}}
\label{proofs:smallgen}

In finite cubefree median graphs -- and squaregraphs, in particular -- there is an 	ample supply of vertices of degrees at most two:

\begin{lemma}
\label{lem:far-deg2}
Let $G$ be a finite cube-free median graph, $v$ be any vertex of $G$, and $w$ be any
vertex that is maximally distant from $v$. Then $w$ has at most two
incident edges. If $G$ is a squaregraph, then $w$ and its neighbors all lie on the exterior face of $G$.
\end{lemma}

\begin{proof}
Suppose to the contrary that there are three incident edges
$aw$, $bw$, and $cw$. Then, as $a$ and $b$ are at distance two apart and both are closer to $v$ than is $w$, the median $x=m(a,b,v)$ is adjacent to both $a$ and $b$ and distinct from $w$. A symmetric argument applies to the two other medians $y=m(a,c,v)$ and $z=m(b,c,v)$. However, no two of these medians can equal each other, for if two of them were equal they would be adjacent to all three of $a$, $b$, and $c$, contradicting the fact that this triplet has $w$ as its unique median. Similarly, as $x$, $y$, and $z$ are at distance two from each other, their median $m(x,y,z)$ must be adjacent to all three. But then these eight vertices $w$, $a$, $b$, $c$, $x$, $y$, $z$, and $m(x,y,z)$ would induce a cube, contradicting the assumption that the graph is cube-free. This contradiction shows that $w$ can have at most two incident edges. In a squaregraph inner vertices must have higher degree, so $w$ must be exterior.
\end{proof}

\bigskip\noindent
{\bf Proof of Proposition~\ref{prop:median_hull}:} We begin with the first assertion of the proposition, that every finite squaregraph $G$ has at least $\min\{4,\#E\}$ vertices of degree one or two. If $G$ is a path, it has $\#E+1$ vertices, all of which have degree one or two. If $G$ is a tree with exactly three vertices of degree one, then it has $\#E$ vertices of degrees one or two. For trees with more than three vertices of degree one, the claim is trivially satisfied from the bound 4.  And if it is not a tree, it has a nontrivial block $B$, which by Lemma~\ref{4corners} has at least four degree-two vertices. Each of these vertices $v$ is either of degree two in $G$ as a whole, or is an articulation point connecting $B$ to other blocks of $G.$ In the latter case, select a neighbor $w$ of $v$ outside $B$ and take a maximal interval $I(v,z)$ containing $w.$ By Lemma \ref{lem:far-deg2}, the vertex $z$, which belongs to some block $B_v$ different from $B,$ has degree at most two.
Thus, each of the degree-two vertices of $B$ is associated with a distinct vertex with degree one or two in $G$, so in this case there are at least four vertices in $G$ with degree one or two.

The next assertion of the proposition is that the vertices of degree at most three in a finite squaregraph $G$ are the articulation points of degree two or three and the endpoints of maximal convex paths.
Suppose by way of contradiction that some maximal convex
path $P$ in $G$ has an endpoint  $v$ of degree at  least 4. Let
$v'$ be the neighbor of $v$ in $P.$ Since $deg(v)\ge 4,$ $v'$ and
one of the neighbors $w$ of $v$ do not belong to a common inner face
of $G.$  Adjoining $w$ to the path $P$ we will obtain a locally
convex path (i.e., a path for which each 2-subpath is convex), which
is necessarily convex \cite{BaCh_wmg}, yielding a contradiction with
the maximality of $P.$  Therefore all endpoints of maximal convex
paths have degrees at most 3 and hence lie on the boundary~$C.$
Conversely, a vertex $u$ of degree 2 belongs to $C$ and has two
neighbors $v$ and $w$ on $C.$ If $u$ is not an articulation point, then $u$, $v$, and $w$ lie on a 4-cycle, whence $u$ is an endpoint of any maximal convex path $P$ extending the edge $uv$. Now,
let $u$ be a vertex of degree 3, necessarily belonging to $C.$ If all three edges $uv$, $uw$, $ux$ emanating from $u$ are on the boundary $C$, then at most one 4-cycle can contain two of the these edges, whence $u$ is an articulation point. Else $ux$, say, belongs to two inner faces, one containing $uv$ and the other $uw$; then any maximal convex path extending $ux$ must end at $u$. This completes the proof of the second
assertion of the proposition.

To conclude the proof, we must show that, given any vertex $v$ in $G$, it is possible to find three vertices of $X$ having $v$ as their median. If $v$ has degree one or two, it must be a member of $X$, and is the median of itself repeated three times, so we need only consider vertices $v$ with degree three or more.

First assume that $v$ is a vertex of degree at least three on the boundary of $G.$ Then one can extend a convex path along the boundary of $G$ in both directions from $v$ (continuing through 	 articulation points if necessary) to a path eventually reaching end vertices $p$ and $q$ of degree one or two, by virtue of Lemma  \ref{lem:far-deg2}. If $v$ is an articulation point, then one can traverse a convex path along the boundary in a third direction until another vertex $r$ of degree one or two is reached, so that $v$ is the median of $p,q,$ and $r.$ Hence we may assume that $v$ belongs to some 2-connected block $B$ and is not an articulation point of $G.$ Let $u$ and $w$ be the two neighbors of $v$ on the boundary of $B,$ where, say, $u$ 	is in $I(v,p)$ and $w$ in $I(v,q).$ Consider the halfspaces $H_1$ and $H_2$ containing $v$ but not $u$ or $w,$ 	respectively. If $H_1\cap H_2$ is not included in $B$ (because it harbors some articulation point of $G$), then 	choose a vertex $r$ from $(H_1\cap H_2)\setminus B$ at maximal distance to $v.$ Then $r$ has degree at most two (by virtue 	of Lemma \ref{lem:far-deg2}) such that $v$ is the median of $p,q,$ and $r.$ Consequently, $H_1\cap H_2$ is a convex subset of $B.$ If 	$H_1\cap H_2$ is a path, which is necessarily a maximal convex path with $v$ as one end vertex, then it meets $X$ in some vertex $x$ by the hypothesis, whence $v$ is the median of $p, q,$ and $x.$ Therefore we may 	finally assume that $H_1\cap H_2$ is contained in $B$ but is not a path. Pick a vertex $t$ in $H_1\cap H_2$ that has 	 neighbors in both $V\setminus H_1$ and $V\setminus H_2$ such that the distance of $t$ to $v$ is as large as possible. Then every vertex in $I(v,t)$ is adjacent to its gates in $V\setminus H_1$ and $V\setminus H_2.$ Now choose a vertex $r$ from $H_1\cap H_2$ 	such that $I(v,r)$ includes $I(v,t)$ where the distance to $v$ is as large as possible. Then $r$ has degree at 	most two within $H_1\cap H_2$ by Lemma \ref{lem:far-deg2} and can have at most one neighbor in $B$ outside $H_1\cap H_2.$ If $r$ is 	adjacent to its gate in $V\setminus H_1$, say, then $I(v,r)$ is a path, whence the degree of $r$ in $B$ (as well as in $G$) is 	two. Else, $r$ has the same degree two in $G$ as in $H_1\cap H_2.$ Hence $v$ is the median of $p, q,$ and $r$ in this 	case.

\begin{figure}[t]
\centering\includegraphics[width=2.5in]{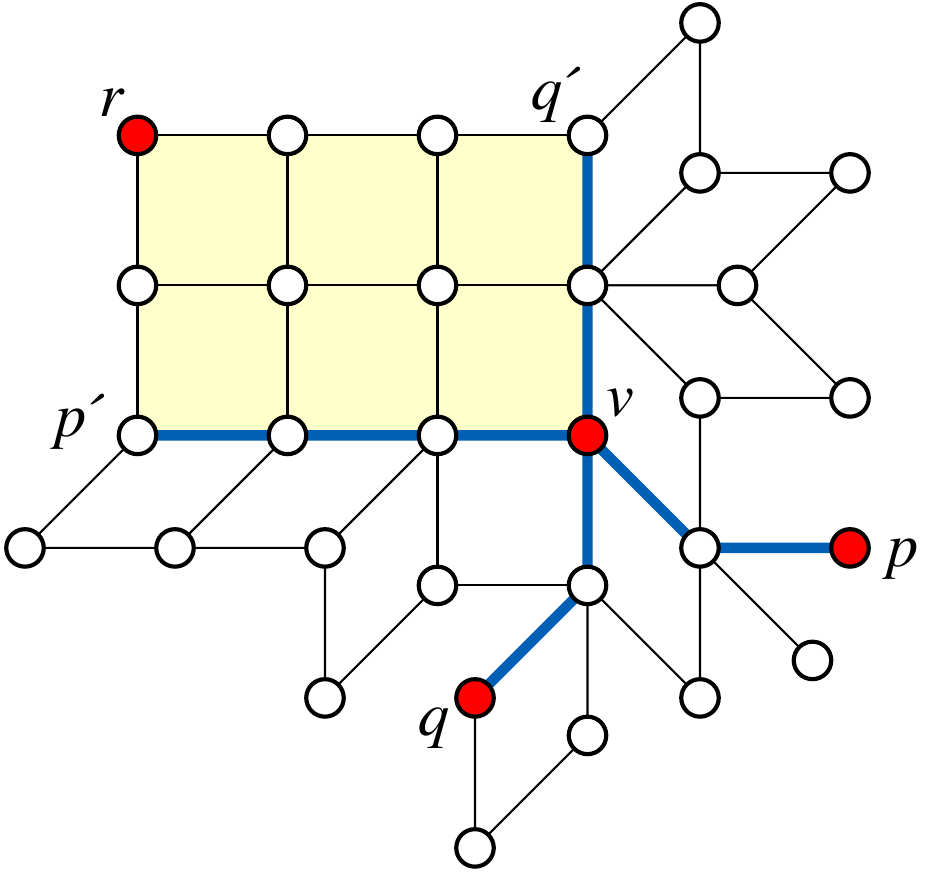}
\caption{Figure for the proof of Proposition~\ref{prop:median_hull}. The two convex paths $P'$ and $Q'$, portions of maximal convex paths $P$ and $Q$ containing points $p$ and $q$ in $X$, partition the squaregraph into two parts, one of which (shaded) contains neither $p$ nor $q$. Letting $r$ be any vertex of degree one or two within the shaded part forms a triplet $(p,q,r)$ of points in $X$ whose median is $v$.}
\label{fig:median-hull}
\end{figure}

Finally, suppose that $v$ is an inner vertex of $G$ and every maximal convex path through $v$ contains a vertex of $X$ only on one side of $v$. Let $p$ be a vertex of $X$ that belongs to one maximal convex path $P$ through $v$. There are two inner faces of $G$ that contain $v$ and are bounded by an edge of path $I(v,p)$; let $vy$ and $vy'$ be the two edges of $G$ that contain $v$ and belong to these two faces, but do not belong to the path. Let path $Q$ be a maximal convex extension of $yvy'$, and let $q$ be a vertex of $X$ belonging to $Q$, as shown in Figure~\ref{fig:median-hull}. $P$ and $Q$ intersect only at~$v$. Partition each of $P$ and $Q$ into two paths at $v$, and let $P'$ and $Q'$ be the subpaths that do not contain $p$ and $q$ respectively. Let $B$ be the block of $G$ containing $v$. Then $P'$ and $Q'$ must end at degree-three boundary vertices of $B$; for, these paths cannot contain any vertex of $X$ by assumption, and if instead one of these paths passed through an articulation point of $G$, one could (as in the proof of the first assertion of this proposition) find an associated vertex of degree one or two in another block of $G$ such that the continuation of the path from the articulation point to the associated vertex remains convex, contradicting the assumption that $v$ does not lie on any convex path between two vertices of $X$. Therefore, the union of the two paths $P'$ and $Q'$ forms a path that partitions the
squaregraph $G$ into two subgraphs that are themselves squaregraphs; let $G'$ be the one of these two subgraphs that does not contain $p$ nor $q$. By the first assertion of the proposition, $G'$ has at least four vertices of degree one or two. Three of these vertices may be $v$ and the endpoints of $P'$ and $Q'$, which do not belong to $X$, but the fourth vertex, $r$, must have degree one or two in $G$ as well and therefore must belong to $X$. As we now argue, the median of $p$, $q$, and $r$ is $v$. First, note that 	the concatenation of two convex paths is always a shortest path because their common endpoint 	serves as the mutual gate for the two convex paths.  Thus, $v$ lies on a shortest path from $p$ to $q$. Consider the gates of the vertex $r$ in $P$ and $Q.$ If these gates
belongs to the subpaths $P'$ and $Q',$ respectively, then necessarily $v$ belongs to the intervals $I(r,p)$ and $I(r,q),$ and we are done. Otherwise, if the gate of $r$ in $P$ belongs to $P\setminus P',$ then the shortest path from $r$ to the gate has to pass through $Q'$ because  $P\cup Q$ is a separator of $G.$ But then $v$ would have to be the gate, giving a contradiction.

%$R$ be any shortest path from $r$ to $v$; by convexity, $R$ lies entirely within $G'$. Because it is a shortest path, $R$ cannot itself cross any convex split of $G$ twice, and neither can the concatenation of $R$ with the convex path from $v$ to $p$ cross such a split twice. For, if it did, the zone described by Lemma~\ref{strip} for the twice-crossed split would have to cross both the path from $v$ to $p$ and the convex path $Q$ that separates $p$ from $G'$, and this zone together with the two paths $P$ and $Q$ would again surround a region with only three degree-two boundary vertices, giving a contradiction. Thus, $v$ lies on a shortest path from $p$ to $r$, and by a symmetric argument it lies on a shortest path from $q$ to $r$. Therefore, it is the median of $p$, $q$, and $r$.

\bigskip\noindent
{\bf Proof of Proposition~\ref{prop:median_hull_bis}:} The first assertion of the proposition follows in one direction from the definition 	of inner lines. As for the converse, if the intersection of two halfspaces is a path lying entirely in some 2-connected block, then this intersection equals the intersection of the zones corresponding to the respective splits and thus constitutes an inner line.

The vertices of degrees one and two in $G$ must belong to every median-generating subset $X$ of $G.$ If such a set $X$ was disjoint from the intersection of two distinct halfspaces $H_1$ and $H_2$ of $G$ that constitutes a maximal convex path $P$ of some block, then $V\setminus P=(V\setminus H_1)\cup (V\setminus H_2)$ is closed under taking medians, whence $X$ cannot median-generate $G.$

Conversely, assume that $X$ meets the criterion. Let $v$ be a boundary vertex of degree at least three in $G.$ Then as in the proof of Proposition \ref{prop:median_hull}, $v$ can be obtained as the median of three vertices from $X.$ Thus, the boundary of $B$ is obtained in one step of median generation from the vertices of degrees at most two. By Proposition \ref{prop:median_hull}, every vertex of $B$ is the median of three vertices from the boundary of $B.$ We conclude that two steps of median generation starting from $X$ suffice to retrieve the entire squaregraph $G.$

\bigskip\noindent
{\bf Proof of Theorem~\ref{theorem:npc-gen}:}
The problem belongs to NP, since we may test any potential $g$-element set $X$ by
generating medians of triplets of elements from $X$ and from the previously-generated elements until finding a subset of $G$ closed under median operations; $X$ is median-generating if this subset consists of all vertices in $G$. To prove NP-hardness, we describe a polynomial time many-one reduction from a known NP-complete problem, edge clique cover on graphs with polynomially many cliques. To be more specific, in the problem we reduce from, we are given a connected undirected graph $F$ with $n$ vertices which is guaranteed not to have any complete subgraphs on five or more vertices, and a number $k$, and we must determine whether $F$ contains a set of $k$ cliques (simplices or complete subgraphs) that together cover all of its edges. Rosgen and Stewart~\cite{RosSte-DMTCS-07} provide a reduction showing that this special case of the clique cover problem is NP-complete.

There is a canonical construction of median graphs generated from arbitrary
graphs \cite{BaVdV}: namely, for a graph $F$ the simplex graph $\kappa(F)$ has the simplices of $F$ as its vertices and pairs of (comparable) simplices differing by the presence or absence of exactly one vertex as its edges.
Given the input graph $F$ with $n$ vertices and $m$ edges, and given the number $k$, we form a graph $F'$ by adding to $F$ a new degree-one vertex adjacent to each vertex of $F$, let $G=\kappa(F')$, and let $g=n+k$. This is a polynomial time transformation: due to the requirement that $F$ have no $K_5$ subgraph, the number of vertices in $\kappa(F')$ is $O(n^4)$. We claim that $G$ has a $g$-element median-generating set $X$ if and only if $F$ can be covered by $k$ cliques.

In one direction, from a clique cover of $F$, we form a median-generating set $X$ by including one element for each of the cliques in the cover, and one element for each of the $n$ edges added to $F$ to form $F'$. Then any edge $e$ of $F$ can be generated as the median of three simplices: the simplex that covers $e$ and the two edges connecting the endpoints of $e$ to the adjacent degree-one vertices of $F'$. Any triangle of $F$ can be generated as the median of its three edges, and similarly any larger simplex can be built up as medians of smaller subsimplices. The empty simplex can be generated as the median of three disjoint edges connecting to degree-one vertices of $F'$. And any single-vertex simplex can be generated as the median of two edges that contain it and the empty simplex. Thus, $X$ is a median-generating set of the required size.

In the other direction, suppose $X$ is a median-generating set of size at most $n+k$. No edge of $G$ can be generated as the median of zero- and one-element simplices, so each edge of $F$ must be a subset of one of the simplices in $X$. But the only simplices in $F'$ that contain one of the edges incident to the newly-added degree-one vertices are those edges themselves, so each of these $n$ edges must form one of the elements of $X$. The remaining $k$ elements of $X$ must cover the remaining edges of $F$, so they form a clique cover of the required size.

This polynomial-time many-one transformation completes the desired NP-completeness proof.

\bigskip\noindent
{\bf Proof of Proposition~\ref{prop:injective-hull}:}
First observe that for any minimal parity-integer metric form $f$ on an absolute retract $G$ of bipartite graphs there exists some vertex $u$ of $G$ with $f = d_u$. Indeed, one can extend $G$ to a bipartite graph $H$ by adjoining a new vertex $z$ and connecting $z$ to each vertex $v$ of $G$ by a new path of length $f(v)$. Since the image $u$ of $z$ under a retraction from $H$ to $G$ satisfies $d_u(v) = d_G(u,v) \le d_H(z,v) = f(v)$ for all vertices $v$ of $G$, the equality $d_u = f$ holds by minimality of $f$. In particular, this holds for every finite cube-free median graph $G$ because these graphs are known to be absolute retracts of bipartite graphs (Corollary 4.5 of~\cite{BaDr86}).

Finite cube-free median graphs have an ample supply of vertices of degree less than three: expand any isometric path between two distinct vertices $u$ and $v$ to a maximal isometric path from $x$ via $u$ and $v$ to $y$. Then by Lemma~\ref{lem:far-deg2} either endpoint can have at most two neighbors in the interval $I(x,y)$ and no neighbors beyond because of maximality. Therefore the particular parity-integer metric forms $d_u$ and $d_v$ differ on the set $X$ of vertices with degrees one or two, namely $d_u(x) < d_v(x)$ and $d_u(y) > d_v(y)$, whence $||d_u - d_v||_\infty \ge d(u,v)$. The reverse inequality follows from the fact that $||d_t - d_w||_\infty \ge 1$ for every edge $tw$ of $G$. Thus, the squaregraph $G$ is isomorphic to the absolute retract of bipartite graphs generated by the metric space $(X,d|_X)$ embedded in $(\mathbb{N}^X,||\cdot||_\infty)$.

\section{Proofs for Section~\ref{sec:local-x}}
\label{proofs:local-x}

To characterize squaregraphs without reference to the plane
embedding, we use the following notation. The  {\it rim}
$R(u)$ at a vertex $u$ consists of the neighborhood $N(u)$ of
$u$ and all vertices other than $u$ adjacent to at least two vertices of $N(u).$
The {\it closed rim} $R[u]$ is the rim $R(u)$ together with its hub $u.$
A median graph is \emph{cube-free,} that is, it does not have an
induced (or isometric) 3-cube $Q_3$ exactly when it does not include
a 3-cogwheel, that is, each $R[u]$ is a median subgraph. The
vertices of a squaregraph are then classified into inner and
boundary vertices: a vertex $u$ is an inner vertex exactly when
$R[u]$ is a cogwheel. The vertex $u$ is on the boundary if and only
if either $R[u]$ is a bipartite fan, alias {\it cogfan} 	(that is,
the rim $R(u)$ induces a path) or $u$ is an articulation point of $R[u]$ (in which
case the squaregraph is not 2-connected).

\begin{lemma} \label{2connected} The following statements are equivalent for a finite
median graph $G:$

\begin{itemize}
\item[(i)] $G$ has a plane embedding as a 2-connected squaregraph;
\item[(ii)] all closed rims in $G$ are cogfans or cogwheels;
\item[(iii)] $G$ is 2-connected and cube-free such that
$K_2\Box K_{1,3}$ is not an induced (or isometric) subgraph.
\end{itemize}
\end{lemma}

\begin{proof} From the remark preceding this result we conclude that
(i)$\Rightarrow$(ii). If the closed rims of $G$ are cogfans or
cogwheels, then $G$ cannot contain articulation points, thus $G$
is 2-connected. Moreover, the cube $Q_3$ and $K_2\Box K_{1,3}$ are
forbidden as well, showing that (ii)$\Rightarrow$(iii).  To
establish that (iii)$\Rightarrow$(i), proceed by induction.   Select
a convex split $\sigma(uv)=\{ W(u,v),W(v,u)\}$ of $G$ such that
$W(u,v)$ is minimal by inclusion among all such halfspaces. We
assert that $W(u,v)$ coincides with $P(u,v)=Z(uv)\cap W(u,v)$, i.e., every vertex of
$W(u,v)$ contains a neighbor in $W(v,u)$. Suppose by way of
contradiction that there is a vertex $x\in W(u,v)\setminus P(u,v)$.
Since $W(u,v)$ is convex, we can select $x$ to be adjacent to a
vertex $y\in P(u,v)$. Consider the convex split $\sigma(xy)$. If
$W(x,y)\subset W(u,v)$, we obtain a contradiction with the
minimality choice of $W(u,v)$. Thus the splits $\sigma(uv)$ and
$\sigma(xy)$ are incompatible. This means that $W(x,y)$ shares a
vertex with $W(v,u)$. Then necessarily there is a vertex $z\in
W(x,y)\cap P(u,v)$. Since $y$ is adjacent to $x$ and belongs to
$P(u,v)$, we deduce that $y\in I(x,z)$. Since $y\in W(y,x)$, we
obtain a contradiction with the convexity of $W(x,y)$. Hence indeed
$W(u,v)=P(u,v)$. Since $G$ is $Q_3$- and $K_2\Box K_{1,3}$-free,
necessarily $P(u,v)$ constitutes a convex path. Its neighbors
outside form an isomorphic path such that each edge of the latter
lies on the boundary of the graph $G\setminus P(u,v)$ which can be realized as  a squaregraph
by virtue of the induction hypothesis. Then $G$ is a squaregraph as
well.
\end{proof}

\begin{lemma}\label{directed-union} The following statements are equivalent for an infinite median graph $G:$
\begin{itemize}
\item[(i)] $G$ is 2-connected and cube-free as well as $K_2\Box K_{1,3}$-free;
\item[(ii)] $G$ is the directed union of convex subgraphs that are 2-connected finite squaregraphs;
\item[(iii)] all closed rims in $G$ are cogfans or cogwheels.
\end{itemize}
\end{lemma}

\begin{proof} The implication (iii)$\Rightarrow$(i) is obvious and the implication (ii)$\Rightarrow$(i) follows
from Lemma  \ref{2connected}. To show that (i)$\Rightarrow$(ii), we start with any inner face of $G$ as the
first 2-connected
convex subgraph $H_1.$ Let  $H_1\subset H_2\subset\cdots \subset H_i$ be a chain of
$i$ convex 2-connected  finite squaregraphs of $G$ and let $v$ be a vertex of $G$ outside $H_i$ which is
adjacent to a vertex $u$ of $H_i.$ Since $G$ is 2-connected, the vertex $v$ can be connected by
a path $P$ to some vertex $w\ne u$ of $H_i$ not passing via $u.$
The subgraph induced by $H_i\cup P$ is 2-connected, therefore the subgraph $H_{i+1}$ induced by its
convex hull is also 2-connected. On the other hand, since $H_i\cup P$ is finite and the graph $G$ is median,
this convex hull is also  finite. From Lemma  \ref{2connected} we infer that  $H_{i+1}$ is a 2-connected finite squaregraph.
This establishes that $G$ is the directed union of convex subgraphs that are 2-connected finite
squaregraphs.  It remains to show that (i)$\Rightarrow$(iii). Since $G$ is 2-connected, any rim $R(u)$ is a connected subgraph.
If $R(u)$ is not a cycle or a path, then some vertex $x\in R(u)$ has at least three neighbors $y_1,y_2,y_3$ in $R(u).$ If $x$ is adjacent to
$u,$ then the subgraph induced by $u,x,y_1,y_2,y_3$ is a forbidden $K_2\Box K_{1,3}.$ On the other hand, if $x$ is not adjacent to $u,$
then $y_1,y_2,$ and $y_3$ are adjacent to $u$ and we obtain a forbidden $K_{2,3}.$
\end{proof}

\begin{lemma}\label{amalgam} For a finite cube-free median graph $G$ the following
statements are equivalent:
\begin{itemize}
\item[(i)] $G$ is circular, viz. the traces of the convex splits of $G$
on the boundary $C$ form a circular split system;
\item[(ii)] $G$ does not contain $K_2\Box K_{1,3}$ as an induced (or isometric) subgraph and
does not contain an induced (or isometric) suspended cogwheel;
\item[(iii)] $G$ can
be obtained by successive amalgamation of 2-connected squaregraphs
along boundary vertices.
\end{itemize}
\end{lemma}

\begin{proof} First observe that in a cube-free median graph any subgraph isomorphic to the
4-cycle, or the domino $K_2\Box K_{1,2}$, or $K_2\Box K_{1,3}$ is necessarily locally convex and hence
convex. Consequently, if the latter graph is forbidden as an induced subgraph, then every induced 	
cogwheel is (locally) convex and further every induced suspended cogwheel is convex.
Now suppose that $G$ is circular but had a convex subgraph $H$ isomorphic to $K_2\Box K_{1,3}$ or a
suspended cogwheel. Then the boundary $C$ of $G$ contains a subset $X$ that is bijectively 	
mapped by the gate map from $G$ to $H$ to the set $Y$ of vertices of degree at most two in $H.$
The total cyclic order restricted to $X$ would then enforce a total cyclic order on $Y$ so that
the traces on $Y$ of the halfspaces of $H$ are arcs, which however is impossible. Therefore (i)$\Rightarrow$(ii)
holds.

Conversely, if a cube-free median graph
$G$ does not contain the forbidden configurations, then from Lemma
\ref{2connected} we infer that every 2-connected component of $G$ is
a 2-connected squaregraph. Since $G$ does not contain isometric
suspended cogwheels, the 2-connected components of $G$ can only
intersect in their boundaries, thus (ii)$\Rightarrow$(iii).

From Lemma \ref{strip}, we know that (iii)$\Rightarrow$(i) holds for
2-connected squaregraphs. To establish this implication for all
squaregraphs, we proceed by induction on the number of 2-connected
components of $G.$ Suppose that $G$ is obtained as an amalgamation
of two squaregraphs $G'$ and $G''$ along a common boundary vertex
$x.$ By induction assumption, the graphs $G'$ and $G''$ admit
circular representations $C'$ and $C''$. Let $x'$ and $x''$ be the
copies of $x$ in $C'$ and $C''$, respectively. Let $y'$ be a
neighbor of $x'$ in $C'.$  To obtain a circular representation $C$
of $G,$ we identify the two copies of $x$ in $C'$ and $C'',$ take
the circular representation $C'$ and place the remaining vertices of
$C''$ next to $x'$ but before $y'$ following the circular order of
$C''.$  To see that the traces of convex splits of $G$ on $C$ indeed
form a circular split system notice that any convex split
$\sigma(uv)$ of $G$ can be derived from the convex splits of $G'$
and $G''$ in the following way. Suppose without loss of generality
that $uv$ belongs to $G'$ and denote by $W'(u,v)$ and $W'(v,u)$ the
halfspaces defined by the edge $uv$ in $G'.$ If $x\in W'(v,u),$ then
$W(u,v)=W'(u,v)$ and $W(v,u)=W'(v,u)\cup V(G'').$ If $x\in W'(u,v),$
then $W(u,v)=W'(u,v)\cup V(G'')$ and $W(v,u)=W'(v,u).$ From the
construction of $C$ we infer that  the traces of $W(u,v)$ and
$W(v,u)$ on $C$ define indeed a circular split, establishing that
(iii)$\Rightarrow$(i).
\end{proof}

\begin{lemma} If $G$ is an infinite 2-connected cube-free median graph with no induced
subgraph $K_2\Box K_{1,3}$ and no induced suspended cogwheel, then the
traces of convex splits of $G$ on the virtual boundary of $G$
(comprising the endpoints of all finite zonal paths and zonal rays of $G$)
form a circular split system.
\end{lemma}

\begin{proof} By Lemma \ref{directed-union}, $G$ is the directed union
of convex 2-connected finite squaregraphs $H_1\subset H_2\subset\cdots,$
where $H_1$ is an inner face of $G.$ Replace each edge of $G$ by a pair
of opposite arcs and assign each arc to exactly one  incident
face of $G$ in such a way that the arcs assigned to the face $H_1$ define a
counterclockwise traversal of $H_1.$ Then the inner faces of each $H_i$ are also
traversed in the counterclockwise order while the outer face $C_i$ of $H_i$
is traversed in the clockwise order.

Let $P_1,P_2,$ and $P_3$ be three maximal convex paths of $G$. To define a
ternary relation between the endpoints of these paths,
we consider the first squaregraph $H_i$ which is crossed by each of the paths
$P_1,P_2,$ and $P_3$ and which contains all eventual pairwise intersections of these
paths. Let $P_k\cap C_i=\{ u'_k,u''_k\}$  ($k=1,2,3$). Then we define the ternary relation
on the vertices $u'_1,u''_1,u'_2,u''_2,u''_3,u''_3$ as induced by the clockwise
orientation of the cycle $C_i.$ For any $j>i,$ the six intersections of the paths
$P_k$ $(k=1,2,3)$ with $C_j$ induce the same ternary relation because the intersection
pattern of the  paths $P_1,P_2,$ and $P_3$ is the same in $H_i$ and $H_j$ and because
$C_i$ and $C_j$ have the same orientation. Let $\beta$
denote the resulting ternary relation on the set $C$ of endpoints and virtual endpoints of maximal
convex paths of $G.$ Then $\beta$ satisfies the Huntington's axioms because these axioms are defined
on four points and $\beta$ is a total cyclic order on each $C_i$. Thus $\beta$ is a total
cyclic order on $C.$
\end{proof}

Note that, in particular, any tree has a circular split system. In
this case it suffices to take the set $X$ of endpoints. The
characteristic property of the circular split systems derived from a
squaregraph is that the underlying circle is uniquely determined
(for an illustration, see Figure~\ref{fig:dual}). The canonical set $X$
of circular cube-free median graph determining the convex splits would consist
of all endpoints (if any) and all boundaries of squaregraph
blocks minus articulation points. The general duality between median
algebras and binary data tables (Isbell's binary messages) could
then be formulated in terms of canonical boundary subsets $X$ in the
case of circular cube-free median graph.

%\begin{lemma} \label{trace}   Two
%convex splits of a finite squaregraph  $G$ are incompatible exactly when their traces on
%any subset $X$ of $C$ which median-generates $G$  are incompatible.
%Hence the triangle-free incompatibility graphs ${\mathrm I}{\mathrm
%n}{\mathrm c}({\mathcal S}(G))$ and ${\mathrm I}{\mathrm n}{\mathrm
%c}({\mathcal S}(G)|_X)$ are isomorphic.
%\end{lemma}
%
%\begin{proof} It remains to note that by Proposition \ref{strip} every
%zone $Z(uv)$ intersects $C$ exactly in $uv$ if $uv$ is a bridge of
%$G$ and in the two terminal spokes of the ladder $Z(uv)$ otherwise.
%Thus, the trace of $\sigma(uv)$ on $C$ induces a circular split. If
%$\sigma=\{A,B\}$ and $\sigma'=\{ A',B'\}$ are incompatible splits of
%$G$ with compatible traces on $X,$ then $B\cap B'\cap X,$ say, is
%empty. By Proposition~\ref{prop:median_hull}, any vertex $z$ of $B\cap B'$ is the
%median of three vertices from $X\subseteq A\cup A'$ and therefore
%must belong to $A\cup A'$ (because $A$ and $A'$ are halfspaces),
%which is however disjoint from $B\cap B',$ a contradiction.
%\end{proof}

\bigskip\noindent
{\bf Proof of Proposition \ref{prop:forbidden}:} The implication (i)$\Rightarrow$(ii)
follows from Lemma \ref{median}. The converse implication
(ii)$\Rightarrow$(i) is a consequence of Lemmas \ref{2connected} and
\ref{amalgam}. The equivalence (ii)$\Leftrightarrow$(iii) is
obvious. Finally, the equivalence between the conditions (ii) and
(iv) follows from Lemma \ref{amalgam} and the fact mentioned in Section~\ref{sec:smallgen}
that the
incompatibility graph Inc$({\mathcal S}(G))$  of convex splits of a
squaregraph $G$ is isomorphic to the
incompatibility graph Inc$({\mathcal S}(G)|_X)$ of traces of convex
splits on any median-generating subset $X$ of $C.$

\bigskip\noindent
{\bf Proof of Proposition \ref{prop:forbidden-infinite}:} To every finite set $W$ in a median
graph $G$ add some shortest path for each pair of vertices from $W.$
Then the extension $X$ of $W$ is finite and hence its convex hull is
a finite median algebra, which is a median subgraph of $G$ because
$X$ induces a subgraph of $G.$ Then the implications
(i)$\Rightarrow$(ii)$\Rightarrow$(iii)$\Rightarrow$(i) and
(iv)$\Rightarrow$(iii) are obvious. Assume that (iv) is violated,
then select a finite subset $U$ that testifies to a forbidden
intersection pattern of the halfspaces. Moreover, for each
halfspace, add to $U$ a pair of adjacent vertices separated by that
halfspace. Then expand this extended set $W$ further to a finite set
$X$ inducing a connected subgraph of $G.$ The convex hull $F$ of $X$
is then a finite median graph where all involved halfspaces of $G$
leave their distinct traces with the same kind of forbidden
intersection pattern, that is, (iii) is violated. Moreover, suppose
there was a total cyclic order on which the system of convex splits of $G$
becomes circular. Then a finite subset $C$ of this cyclic order
would faithfully represent the intersection pattern of the
halfspaces of $F.$ However, three pairwise disjoint halfspaces would
leave a trace of three pairwise disjoint cycle segments, which
cannot be met simultaneously by a single cycle segment. Moreover, a
cyclic intersection pattern of halfspaces yields a cyclic
intersection pattern of cycle segments which cover the entire cycle,
so that there is no space for a proper segment to cover all those
segments. Therefore (v) must be violated. This shows that (v) implies (iv). 
Finally, to show that (ii) implies (v), it suffices to note that for
the directed union the finite total cyclic orders consistently go along with
it, so that we can represent the system of convex splits on the
directed union of the corresponding cyclic orders.

\section{Proofs for Section~\ref{sec:line-arrangements}}
\label{proofs:line-arrangements}

In order to prove Theorem ~\ref{theorem:infinite} we first consider the balls in finite squaregraphs.

\begin{lemma}
\label{lem:Gvr-is-square}
Let $G$ be a finite squaregraph, $v$ be any vertex of $G$, and $r$ be any nonnegative integer.
The ball $B_r(v)$ induced by the vertices that are at distance at most $r$ from $v$ in $G$ is a squaregraph.
\end{lemma}

\begin{proof}
If $G=B_r(v)$, the result is clearly true. Otherwise, fix an embedding of $G$ as a squaregraph and let $w$ be a vertex of $G$ that is maximally distant from $v$; $B_r(v)\subset G\setminus\{w\}\subset G$.
By Lemma~\ref{lem:far-deg2}, each inner vertex of $G\setminus\{w\}$ is also an inner vertex of $G$ with the same set of neighbors, and each inner face of $G\setminus\{w\}$ is also an inner face of $G$, so $G\setminus\{w\}$ is a squaregraph and the result follows by induction on the number of vertices in~$G$.
\end{proof}

\begin{figure}[t]
\includegraphics[width=3in]{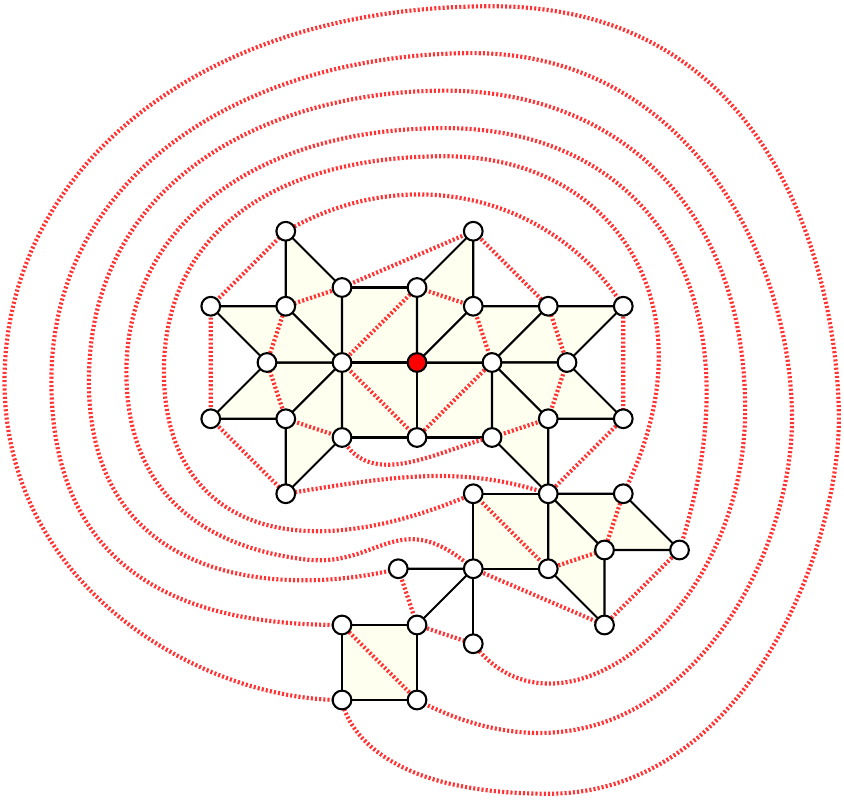}
\caption{Illustration of Lemma~\ref{lem:concentric}: for any finite squaregraph $G$ the vertices may be connected into concentric cycles according to their distances from an arbitrarily chosen starting vertex, and the cycles embedded in the plane consistently with $G$.}
\label{fig:concentric}
\end{figure}

\begin{lemma}
\label{lem:concentric}
Let $G$ be a finite squaregraph, and let $v$ be any vertex of $G$. For each positive integer $i$ let $S_i(v)$ be the set of vertices of $G$ that are at distance $i$ from $v$. Then one can find a planar supergraph $H\supset G$, and a planar embedding of $H$ consistent with the embedding of $G$, such that in $H$ each set $S_i(v)$ is connected into a cycle, and such that these cycles are embedded concentrically surrounding $v$.
\end{lemma}

\begin{proof}
We will use induction on the number of vertices of $G$.
Let $w$ be any vertex maximally distant from $v$, and let $G'=G\setminus\{w\}$. As above, $G'$ is a squaregraph, so by the induction hypothesis it has a plane-embedded supergraph $H'\supset G'$ with the properties described in the lemma. As we now show, we may add $w$ back to $G'$, forming $G$, and simultaneously modify $H'$ to form the desired supergraph $H$. There are three cases:

\begin{enumerate}
\item If $w$ is the only vertex of $G$ at distance $d(v,w)$ from $v$, then the outer face of $H'$ consists of the cycle connecting all vertices at distance $d(v,w)-1$, and must include all neighbors of $w$. We can extend $H'$ to $H$ by connecting $w$ to its neighbors by edges that are drawn outside this cycle and then adding a self-loop attached at $w$ that surrounds the rest of the graph including the new edges.

\item If there are other vertices of $G$ at the same distance as $w$ from $v$, and $w$ has a single neighbor $x$ in $G$, then $x$ belongs to the second-from-outermost cycle of $H'$, and there is a unique face $f$ of the embedding of $H'$ that includes both $x$ and an edge of the outer cycle into which $w$ can be added consistently with the embedding of $G$. In this case, we may form $H$ by splitting the outer edge of $f$, placing $w$ at the split point, and adding back the edge connecting $w$ to $x$ within $f$.

\item In the remaining case, there are other vertices of $G$ at the same distance as $w$ from $v$, and $w$ has two neighbors $x$ and $y$ in $G$ that belong to the second-from-outermost cycle. Let $z$ be the median of $x$, $y$, and $v$; $z$ belongs to the third-from-outermost cycle, and by Lemma~\ref{4corners} the four vertices $x$, $y$, $z$, and $w$ form a face of $G$. In the embedding of $G$ (and of $G'$) $x$ and $y$ appear consecutively among the neighbors of $z$, and therefore they also appear consecutively on their cycle. There exists a unique face $f$ of the embedding of $H'$ that contains edge $xy$ and an edge of the outer cycle; as in the previous case we may form the desired embedding of $H$ by splitting the outer cycle edge, adding $w$ at the split point, and routing the edges $wx$ and $wy$ through $f$.
\end{enumerate}
\end{proof}

Figure~\ref{fig:concentric} illustrates Lemma~\ref{lem:concentric}. Note that the lemma cannot be generalized to all planar graphs: the planar graph $K_{2,4}$, with $v$ chosen as one of its degree-two vertices, cannot be augmented to have concentric cycles like those of the lemma.

\begin{lemma}
\label{lem:locally-finite-line-arrangement} For a line arrangement $\mathcal A$ in the hyperbolic plane, the following assertions are equivalent:

\begin{itemize}
\item[(i)]  $\mathcal A$ is locally finite;
\item[(ii)]	$\mathcal A$ is countable and the cell copair hypergraph of $\mathcal A$ is discrete;
\item[(iii)]	$\mathcal A$ is countable and the dual graph of $\mathcal A$ is connected.
\end{itemize}
\end{lemma}

\begin{proof} (i) $\Rightarrow$ (ii): The hyperbolic plane is a countable union of balls $B_n(v)$ around any point $v$ where $n$
runs through the positive integers. Since each $B_n(v)$ meets only finitely many lines of $\mathcal A$, there are no more than
countably many lines in $\mathcal A$ altogether. The line segment between any two points $u$ and $v$ from different cells is
included in some ball and hence meets only finitely many lines from $\mathcal A$. Therefore the cell copair hypergraph of $\mathcal A$ is discrete.
	
(ii) $\Rightarrow$ (iii): The line segment $S$ between any two points $u$ and $v$ from different cells is separated by finitely many
lines from the arrangement. Since $\mathcal A$ is countable, we may assume that $u$ and $v$ are chosen so that every point on the
line segment $S$ meets at most one line from $\mathcal A$. Then removing from $S$ these individual crossing points with the lines
from $\mathcal A$ results in a finite sequence of (half-)open segments that each extend to cells such that two consecutive
cells are separated by exactly one line from $\mathcal A$. Therefore the dual graph is connected.
	
(iii) $\Rightarrow$ (i): If $R$ is a bounded region of the hyperbolic plane, surround $R$ by a polygon $P$; the assumption that $\mathcal A$ is countable
guarantees that $P$ can be chosen so that it does not pass through any crossing point of $\mathcal A$. Each edge of $P$ is crossed only by the lines
of $\mathcal A$ that separate the two endpoints of the edge and that must be crossed by any path between those endpoints; therefore, the number of
lines crossed by the edge is equal to the length of the shortest path in the graph between the vertices dual to the endpoints' cells, and is
finite. Any line of $\mathcal A$ that intersects $R$ crosses one of the finitely many edges of $P,$ so the total number of lines of $\mathcal A$
intersecting $R$ is finite.
\end{proof}

The assumption of countability in (ii) and (iii)  is necessary, as the set of all hyperbolic lines has an empty cell copair hypergraph and empty dual graph.

\begin{lemma}
\label{lem:locally-finite-triangle-free-arrangement} For a locally finite line arrangement $\mathcal A$ in the hyperbolic plane,
the following statements are equivalent:

\begin{itemize}
\item[(i)]  $\mathcal A$ is triangle-free;
\item[(ii)]	the cell copair hypergraph of $\mathcal A$ is $^*$Helly;
\item[(iii)] the dual graph of 	$\mathcal A$ is median.
\end{itemize}
\end{lemma}

\begin{proof} (i) $\Rightarrow$ (ii): We will apply Proposition \ref{prop:Helly*} to the cell copair hypergraph of $\mathcal A$, which is discrete by
the preceding lemma. Let $u,v,w$ be any three points of the hyperbolic plane belonging to different cells of $\mathcal A$. Consider the collection
${\mathcal D}$ of open halfplanes that are bounded by lines from $\mathcal A$ and contain at least two of $u,v,w.$ Then $\mathcal D$ partitions
into ${\mathcal D}_2$ and ${\mathcal D}_3$ of those halfplanes that contain exactly two of $u,v,w$ and all three of them, respectively.
The intersection of ${\mathcal D}_3$ contains $u,v,w$ and is convex. Since the cell copair hypergraph is discrete, ${\mathcal D}_2$
is a finite collection intersecting in pairs and hence in triplets by virtue of (i). Then the finite collection ${\mathcal D}_2\cup\{\cap {\mathcal D}_3\}$
of convex sets triplewise intersects and hence, by Helly's theorem, has a nonempty intersection (which includes a cell of ${\mathcal A}$). This proves
that the cell copair hypergraph is $^*$Helly.

(ii) $\Rightarrow$  (iii): Since the cell copair hypergraph separates the points (cells) by definition and is $^*$Helly,
it is the halfspace hypergraph of a median graph on the same vertex (cell) set. Necessarily, this median graph must be the dual
graph of the arrangement because its edges are characterized in terms of separation by a unique line.

(iii) $\Rightarrow$ (i): From the preceding lemma we know that the dual graph $G$ of $\mathcal A$ is connected such
that any two distinct cells $Y$ and $Z$ are separated by no more than finitely many lines from $\mathcal A$. These lines
and no others cross the line segment between any two points $y$ from $Y$ and $z$ from $Z,$ where we may assume that
$y$ and $z$ are chosen so that no point of the segment between $y$ and $z$ belongs to more than one line of $\mathcal A.$ As in the part (ii) $\Rightarrow$ (iii)
of the proof of Lemma \ref{lem:locally-finite-line-arrangement}, we can then extend the open subsegments between consecutive
crossing lines to cells of $\mathcal A$, which thus constitute the inner vertices of a shortest path between the vertices
$Y$ and $Z$ in the dual graph of $\mathcal A$. This, in particular, shows that every shortest path between cells
traverses a line from $\mathcal A$ at most once. Now suppose, by way of contradiction, that
$\mathcal A$ was not triangle-free, that is, there were three pairwise intersecting open halfplanes bounded by
lines from $\mathcal A$. Select a cell each within each of their pairwise intersections. By what has just been shown,
the median cell of the three selected cells belongs to all three open halfplanes, which is impossible.
\end{proof}

\bigskip\noindent
{\bf Proof of Theorem~\ref{theorem:infinite}:} (i) $\Rightarrow$ (iii): We are given a
locally finite squaregraph embedding $G$, and must show that there exists a countable chain of finite squaregraphs covering $G$.
%The condition that any bounded subset intersects finitely many
%edges implies bounded degree: in particular, any vertex $v$ can have degree no more
%than the number of edges intersecting some particular bounded subset with $v$ in its
%interior.
For each positive integer $i$, let $C_i$ be a circle
centered at the origin with radius $i$. Form the arrangement of curves
consisting of $C_i$ together with all edges of the drawing of $G$ that start
and end within $C_i$, and let $R_i$ be the bounded region complementary to
the single unbounded face of this arrangement. Define $G_i$ to be the
subgraph of $G$ induced by the vertices within $R_i$. Then the induced drawing of
$G_i$ has as its inner faces the inner faces of $G$ that are entirely
contained within $R_i$, together with one unbounded face, so its inner
faces are all quadrilaterals, and any vertex that is entirely
surrounded by inner faces has degree four or more. Thus, $G_i$ is a
finite squaregraph, and the sequence of squaregraphs formed for
increasing values of $i$ form a chain that covers $G$.

(iii) $\Rightarrow$ (ii): We are given a countable chain of finite squaregraphs
covering $G$, and must show that each fixed-radius induced subgraph $B_r(v)$ is also a squaregraph.
If the chain covers $G$, some squaregraph $G_i$ in the chain covers
all of the (finitely many) vertices and edges in $B_r(v)$. Applying Lemma~\ref{lem:Gvr-is-square}
to $G_i$ and its subgraph $B_r(v)$ shows that $B_r(v)$ is also a squaregraph.

(ii) $\Rightarrow$ (i): We are given a graph $G$  in which each fixed-radius induced subgraph
$B_r(v)$ is a finite squaregraph, and must show that $G$ has a locally finite squaregraph embedding.
We assume that $G$ is infinite, for otherwise the result follows trivially. Fix some particular vertex
$v$ of $G$. We form an infinite tree, the nodes of
which represent equivalence classes of embeddings of $B_i(v)$ for some $i$, together
with a cyclic augmentation $H_v^i$ as described by Lemma~\ref{lem:concentric}; two augmented embeddings are
considered to be equivalent if they have the same value of $i$ and the nodes in each of their cycles appear in the
same cyclic order. We connect these nodes into a tree by making the parent of a node $H_v^i$ be the node
$H_v^{i-1}$ formed by deleting the outer cycle from $H_v^i$. Each node of this tree has finite degree
(since there are finitely many ways in which one might choose a cyclic permutation of the finitely
many vertices in the next layer of $G$) but the tree has infinitely many nodes, so by K\"onig's
infinity lemma there exists an infinite path in the tree, giving us a choice of cyclic ordering
of all the layers of $G$ that is consistent with a planar embedding of $G$. To transform
this combinatorial embedding into an actual embedding of $G$, place the
vertices that are at distance $i$ from $v$ onto a circle of radius $i$ centered at the
origin, with the edges between vertices at distance $i-1$ and distance $i$
placed within an annulus between two such circles. Each of these annuli contains finitely
many vertices and edges, so the whole embedding formed in this way is locally finite as required.

(i) \& (ii) $\Rightarrow$ (iv): We are given a locally finite squaregraph embedding of
$G$ and must show that $G$ is a connected component of the dual graph of a triangle-free hyperbolic line arrangement.
It is convenient to use the Klein model of the hyperbolic plane, in which the plane is modeled
by a Euclidean unit disk and a hyperbolic line is modeled by a line segment connecting two
points on the Euclidean unit circle. We will use the embedding of $G$ to choose a cyclic order
for which the splits of $G$ form a circular split system, and form a line for each split such
that the lines have distinct endpoints and induce the same split system on the unit circle.
By Lemma~\ref{strip} and property (ii), the splits of $G$ correspond to zones in the form
of (possibly infinite) ladders. By choosing arbitrarily an orientation for each zone we may
distinguish its two ends, corresponding to the two endpoints of the Klein model line segment
we will construct to represent a hyperbolic line. The given embedding of $G$ determines a
cyclic ordering on the ends of the zones, which may be computed as follows: given three ends
of zones $a$, $b$, and $c$, choose a vertex $v$ of $G$ and a radius $r$ sufficiently large
that the ball $B_r(v)$ includes any of the quadrilaterals where two of the zones cross each
other; then each end of each zone corresponds to a boundary edge of $B_r(v)$, and the cyclic
ordering of the edges around the boundary of $B_r(v)$ determines the cyclic ordering of the
ends of zones. It is straightforward to verify by induction on $r$ that this ordering remains
unchanged for larger radii. From the cyclic ordering, we define an arrangement $\mathcal A$
as follows: order the zones of $G$ arbitrarily (e.g., by the closest distance of an edge of
the zone from some fixed vertex $v$), and choose endpoints on the unit circle of the Klein
model for each zone in this order. For the first zone, choose two diametrally opposite endpoints,
and then subsequently place each zone's endpoints consistently with the cyclic order, midway
between the endpoints of its two neighbors in the cyclic order. The result is a hyperbolic
line arrangement $\mathcal A$ with one line per zone in which the cyclic order of the infinite
ends of the hyperbolic lines is the same as the cyclic order of ends of zones. Thus, both the
splits of $G$ and the halfspaces of $\mathcal A$ form isomorphic copies of a circular split
system $\mathcal S$. This implies that two lines of $\mathcal A$ intersect
if and only if they represent two crossing zones of $G.$ Moreover, if we pick any line  $l$ of $\mathcal A$
and consider the lines of $\mathcal A$ intersecting $l,$ then they intersect $l$ in the same
order as the zone represented by $l$ is crossed by the zones represented by these lines.

We observe next that $\mathcal A$ is triangle-free. For, if $\mathcal A$ contained a triangle,
the corresponding three zones of $G$ would have to cross pairwise. But then the cycle formed
by connecting the paths on the outsides of the ladders formed by the three zones would surround
a subgraph of $G$ that is (by local finiteness and Lemma~\ref{4corners}) a finite squaregraph,
but that has only three degree-two boundary vertices (the three corners where the zones cross),
violating Lemma~\ref{4corners}.

For any vertex $v$ of $G,$ the ladder structure of zones and their convexity imply that
two zones incident to $v$ cross in $G$ if and only if they share a common inner face of $G$ incident to $v.$
Therefore, if we traverse synchronously in counterclockwise order the neighbors $w$ of $v$  and the lines
of $\mathcal A$ defined by the zones $Z(v,w)$ (by passing from one line to the next one at their intersection
point), we will traverse a convex region $R_v$ of the unit disk which is the intersection of the halfspaces  to
the left of the traversed lines (which correspond to the halfspaces of $G$ containing $v$; in fact, the intersection
of these halfspaces of $G$ is exactly $v$). We assert that $R_v$
is a face of the arrangement $\mathcal A$. Suppose by way of contradiction that $R_v$ is crossed by a line $l$ of
$\mathcal A.$ Since $\mathcal A$
is triangle-free, either $R_v$ has two supporting lines $l',l''$ of $\mathcal A$ located on both sides of $l$ or all
supporting lines of $R_v$ are located on one side of $l$ except the line $m'$ crossed by $l.$  In the first case we conclude that the zones of
$l'$ and $l''$ belong to
different halfspaces of $G$ defined by the zone of $l.$ This is possible only if this zone is incident to $v,$ i.e., if $l$ is a supporting
line of $R_v.$ In the second case, let $m''$ be the supporting line of $R_v$ intersecting $m'.$ Then the zone represented by $m''$ will separate
the zone represented by $l$ from all zones incident to $v$ except the zone represented by $m'.$ But then the ends of all those zones will define a
cyclic order different from the cyclic order of the ends of the lines representing these zones (i.e., $l$ and the supporting lines of $R_v$).
Hence for each vertex $v$ of $G,$ the region $R_v$  is indeed a face of $\mathcal A$ bordered by finitely many lines of $\mathcal A$, namely, the
lines of $\mathcal A$ representing the zones of $G$  incident to $v.$ The line $l$ representing the zone $Z(v,w)$  borders also the
face $R_w.$ Since two vertices $u$ and $v$ are adjacent in $G$
if and only if there exists a unique zone such that $u$ and $v$
belong to different halfspaces defined by this zone, from the definition of the regions $R_u$ and $R_v$ we conclude that they are separated
by a single line of $\mathcal A$ if and only if  $u$ and $v$ are adjacent. This shows that $G$ is a connected component of the dual graph of
$\mathcal A.$  If the arrangement $\mathcal A$ is locally finite, then  Lemma \ref{lem:locally-finite-line-arrangement} implies that the dual graph
of $\mathcal A$ is connected and therefore it coincides with $G.$ Finally, notice that $G$ is a median graph, as the median of any three vertices $a$, $b$, and $c$ may
be found within the finite squaregraph $B_r(a)$ (where $r={\max\{d(a,b),d(a,c)\}}$) and we have
already seen that finite squaregraphs are median graphs.

%(iv) $\Leftrightarrow$ (v)  This follows from Lemma \ref{lem:locally-finite-triangle-free-arrangement},
%where finite vertex degrees are obviously equivalent to boundedness by finitely many lines of the arrangement.
%This concludes the proof of Theorem~\ref{theorem:infinite}.

(iv) $\Rightarrow$ (i): We are given a  median graph $G$ with finite vertex degrees which is a connected component of the dual graph of
a triangle-free hyperbolic line arrangement $\mathcal A$ satisfying the condition (iv) and
we assert that $G$ has a locally finite squaregraph embedding. Each vertex $v$ of $G$ is represented in $\mathcal A$ by a
convex cell $R_v$ which is bordered by finitely many lines of $\mathcal A$.  We place a dual vertex $v$ within each such cell $R_v,$
choosing a point on each segment of the arrangement where the dual edge is to cross, and connecting
each vertex by line segments to the chosen point on each adjacent segment. Thus, each edge is
represented by a curve formed by two line segments. Each face of the dual contains part of some
line of the arrangement, and each arrangement line is crossed by a sequence of dual edges that
separate each of the crossings on the line from each other, and that separate the crossings
from infinity. Therefore, either a face contains an infinite ray from one or more of the
lines, and is unbounded, or contains a single crossing of two lines and the four dual edges
surrounding that crossing, and is a quadrilateral. This concludes the proof of Theorem~\ref{theorem:infinite}.

\bigskip
We observe that the argument that (iv) $\Rightarrow$ (i) generalizes in a straightforward way to weak pseudoline arrangements (for a definition of such an arrangement for finitely many pseudolines, see \cite[Section 11.4]{EppFaOv}; the definition extends in an obvious way to locally finite arrangements). Therefore, pseudoline arrangements carry no added generality over hyperbolic line arrangements in defining infinite squaregraphs.

\section{Proofs for Section~\ref{sec:tree-products}}
\label{proofs:tree-products}

The first lemma of this section, which characterizes the absence of $2\times 2$ grids in an arbitrary median graph,
covers the case $n = 4$ of Theorems 1 and 2 of \cite{KlKo}.

\begin{lemma} \label{2x2} For a median graph $G=(V,E),$ the following conditions are equivalent:
\begin{itemize}
\item[(i)] $G$ contains a convex $2\times 2$ grid $K_{1,2}\Box K_{1,2};$
\item[(ii)] the incompatibility graph ${\mathrm I}{\mathrm n}{\mathrm c}({\mathcal S}(G))$ includes an induced
4-cycle;
\item[(iii)] the $2\times 2$ grid is a median homomorphic
image of $G;$
\item[(iv)] $G$ includes a median subalgebra $H$ isomorphic
to the $2\times 2$ grid that cannot be expanded to a median
subalgebra isomorphic to $K_{1,2}\Box C_4.$
\end{itemize}
\end{lemma}

\begin{proof} We will verify the implications (i)$\Rightarrow$(ii)$\Rightarrow$(iii)$\Rightarrow$(iv)$\Rightarrow$(i).
%\begin{description}

(i)$\Rightarrow$(ii): If $G$ contains a convex $2\times 2$ grid,
then the four convex splits of $G$ separating this convex grid
evidently form an induced 4-cycle in Inc$({\mathcal S}(G)).$

(ii)$\Rightarrow$(iii): Assume that there is some induced 4-cycle in Inc$({\mathcal S}(G)).$
Then  opposite vertices in this cycle constitute a compatible pair
of splits. Either compatible pair of splits gives rise to a median
homomorphism $\varphi_i$ $(i=1,2)$ onto the 2-path $K_{1,2}$. These
maps factor through a median homomorphism $\varphi$  into the
Cartesian product of these two 2-paths. If $\varphi$ was not
surjective, then necessarily not all four incompatibilities between
the four splits could be maintained.

\begin{figure}[h]
\centering\includegraphics[width=2.5in]{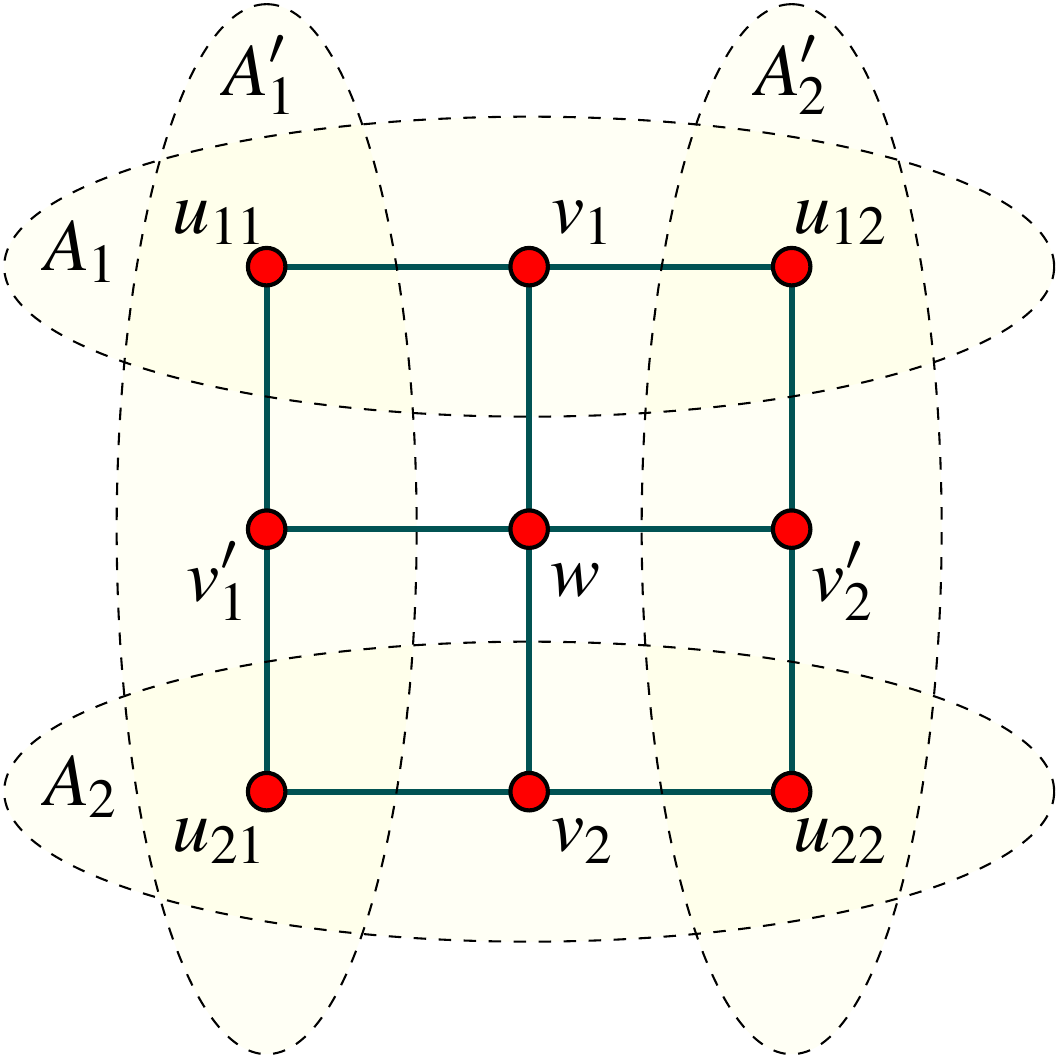}
\caption{To the proof of Lemma~\ref{2x2}.}
\label{fig:2x2}
\end{figure}

(iii)$\Rightarrow$(iv): Next assume that some median homomorphism $\varphi$ maps $G$ onto
the $2\times 2$ grid. Let $\{ A_1,B_1\}$ and $\{ A_2,B_2\}$ be the
pre-images under $\varphi$  of one compatible pair of convex splits
from the grid, and let $\{ A'_1,B'_1\}$ and $\{A'_2,B'_2\}$ be the
pre-images of the other pair, where we may assume that $A_1\cap A_2$
and $A'_1\cap A'_2$ are empty. Those four convex splits of $G$
partition $G$ into nine convex sets; see Figure~\ref{fig:2x2}. Assume that the
distance between  $A_1\cap A'_1$ and $A_2\cap A'_2$ does not exceed
the distance $q$ between $A_1\cap A'_2$ and $A_2\cap A'_1$. Let
$u_{11}$ and $u_{22}$ be mutual gates for the disjoint convex sets
$A_1\cap A'_1$ and $A_2\cap A'_2$  \cite{DrSch}. Then take the gate
$u_{21}$ of $u_{11}$ in $A_2$ and the gate $u_{12}$ of $u_{22}$ in
$A_1.$ Then both $u_{12}$ and $u_{21}$ belong to $I(u_{11},u_{22}).$
Hence the distance between $u_{12}$ of $u_{21}$ must equal $q,$ and
therefore $u_{11}, u_{12}, u_{22},$ and $u_{21}$ form a metric
rectangle with perimeter $2q.$ Consider the ``median set'' $M$ of
vertices of $G$ with minimum distance sum to the four corners of
this rectangle. Then each of the eight halfspaces of $G$ under
consideration includes exactly two corner points, whence the
nonempty intersection $B_1\cap B_2\cap B'_1\cap B'_2$ contains some
vertex $w$ from $M$ \cite{Ba_retract}. Let $v_1, v_2, v'_1$, and
$v'_2$ be the gates of $w$ in $A_1, A_2, A'_1,$ and $A'_2,$
respectively. A straightforward calculation shows that the distance
sum from $w$ to the four corner points is bounded from below by
$2q.$ Since every corner has distance sum $2q$ to the other corner
points, equality holds. Then the nine constructed vertices give rise
to nine overlapping metric rectangles and hence form a median
subalgebra $H$ isomorphic to the $2\times 2$ grid. If $H$ expanded
to a median subalgebra isomorphic to $K_{1,2}\Box C_4,$ then
compatibility of either $\{ A_1,B_1\}$ and $\{ A_2,B_2\}$ or
$\{A'_1,B'_1\}$ and $\{ A'_2,B'_2\}$ would be violated.

(iv)$\Rightarrow$(i):
Assume that $G$ includes some median subalgebra $H$ isomorphic to
the $2\times 2$ grid, with vertices denoted as in Figure~\ref{fig:2x2}, say.
Then select neighbors $x_1, x_2, x'_1,$ and $x'_2$ of $w$ on
shortest paths from the central vertex $w$ to $v_1, v_2, v'_1,$ and
$v'_2,$ respectively. The median $y_{11}$ of $x_1, x'_1$, and
$u_{11}$ is adjacent to $x_1$ and $x'_1,$ thus producing a 4-cycle.
In a clockwise fashion the other three squares are obtained,
altogether yielding a median subgraph isomorphic to the $2\times 2$
grid.
\end{proof}

In the case of cube-free median graph $G$ ``convex'' can be
substituted by ``induced'' in condition (i) of the preceding lemma,
because induced $2\times 2$ grids are necessarily convex in $G.$ For
a similar reason, a median subalgebra $H$ of $G$ isomorphic to the
$2\times 2$ grid cannot be expanded to one that is isomorphic to
$K_{1,2}\Box C_4.$

\begin{lemma}\label{2trees} For a finite squaregraph $G,$ the following conditions are equivalent:
\begin{itemize}
\item[(i)] $G$ can be embedded into the Cartesian product of two trees;
\item[(ii)] Inc$({\mathcal S}(G))$ is bipartite;
\item[(iii)] $G$ does not contain an induced (or isometric, or convex) odd cogwheel;
\item[(iv)] every inner vertex of $G$ has even degree.
\end{itemize}
\end{lemma}

\begin{proof} Conditions (i) and (ii) are equivalent according to \cite{BaVdV}. Further, (iii) and 	
(iv) are also equivalent. If $G$ includes some odd $k$-cogwheel, then Inc$({\mathcal S}(G))$ 	
contains an induced odd $k$-cycle, thus proving (ii)$\Rightarrow$(iii). To prove the converse 	
we may assume that $G$ is 2-connected. If $G$ has some inner vertex $w$ of odd 	degree $k,$
then Inc$({\mathcal S}(R[w]))$ is an odd $k$-cycle. Suppose by way of contradiction  that all
inner vertices of $G$ have even degrees but Inc$({\mathcal S}(G))$ 	contains an induced
odd cycle $C$. Consider a chord diagram representation of the squaregraph  $G$ (within the unit disk) provided by
Theorem~\ref{theorem:infinite}. The chords representing two incompatible splits (i.e., adjacent vertices) from $C$ intersect
in a single point and each such chord has two such intersection points. The segments having these
points as endpoints together define a closed nonintersecting
polygonal line $L(C)$.  Suppose that among all odd cycles of
Inc$({\mathcal S}(G))$, the closed polygonal line $L(C)$ of the selected cycle $C$ has the smallest perimeter.
Since $G$ does not contain odd wheels, the region of the plane bounded by $L(C)$ contains two adjacent inner vertices $u,v$ of $G.$
The chord representing the zone
$Z(uv)$ crosses the chords representing two zones that participate in the cycle $C.$ As a result, that chord (separating $uv$) together
with $C$ gives rise to two cycles $C_1$ and $C_2$ of  Inc$({\mathcal S}(G))$ for which both $L(C_1)$ and $L(C_2)$ have smaller perimeter than $L(C).$
Since one of these cycles is odd, we obtain a contradiction with the choice of the cycle $C.$
\end{proof}

\bigskip\noindent
{\bf Proof of Theorem~\ref{theorem:five_trees}:} We have mentioned in Section~\ref{sec:smallgen}
that the
incompatibility graph Inc$({\mathcal S}(G))$  of convex splits of a
squaregraph $G$ is triangle-free and it is isomorphic to the
incompatibility graph Inc$({\mathcal S}(G)|_X)$ of traces of convex
splits on any median-generating subset $X$ of $C.$ By Theorem~\ref{theorem:infinite},
Inc$({\mathcal S}(G)|_X)$ is isomorphic to the
intersection graph of chords in the circle representation of
${\mathcal S}(G)|_X.$  By a result of \cite{Ko}, a triangle-free
circle graph is 5-colorable, thus the graph Inc$({\mathcal S}(G))$
is 5-colorable as well. Therefore, by a result of \cite{BaVdV}, the
squaregraph $G$ is isometrically embeddable into the Cartesian
product of 5 trees. If a squaregraph $G$ does not contain induced
$2\times 2$ grids, then by Lemma \ref{2x2} the graph Inc$({\mathcal S}(G))$ (and
therefore the underlying circle graph) is $C_4$-free. By a result of
\cite{Ag1}, Inc$({\mathcal S}(G))$ is 3-colorable, whence $G$ is
embeddable into the product of 3 trees. Lemma \ref{2trees} covers the case where $G$
can be embedded into the product of 2 trees.

To complete the proof, assume that $G_0$ is a 2-connected squaregraph with $r$ vertices of
degree $>5$ such that $G_0$ cannot be embedded into the Cartesian
product of four trees. We will construct a sequence $G_0,\ldots,G_r$
of squaregraphs preserving 2-connectivity and this embedding
property such that each $G_i$ is a median homomorphic image of
$G_{i+1}$ $(i=0,\ldots,r-1)$ and has at most $r-i$ vertices of
degree $>5.$ Having constructed $G_i,$ assume that $u$ is a vertex
of degree $h>5.$ If $u$ is an inner vertex, then we can select a
convex 2-path $P_u$ in the closed rim $R[u]$ such that the degrees of
$u$ in the two components of $R(u)\setminus P_u$ are $<h-3.$ If $u$ is on the
boundary, one can select a neighbor $v$ of $u$ such that $P_u=uv$
has the analogous property for cogfan  $R[u].$ In either case,
extend $P_u$ to a minimal convex path $P$ that separates $G_i.$ Then
$G_i$ is the amalgam of smaller squaregraphs $G'_i$ and $G''_i$ along $P.$ In either constituent the degree of $u$ is $<h-1.$ Now,
expand $G_i$ along $P$ by introducing a new zone. In the expansion
$G_{i+1}$ the degrees of vertices do not grow when passing from
$G_i$ to $G_{i+1}$. In the case of $u,$ the degree is $<h.$ Thus,
after at most $r$ expansion steps, no vertices of degrees $>5$ are
left.

\bigskip\noindent
{\bf Proof of Theorem~\ref{theorem:five_trees_infinite}:} Let $G$  be a median graph not containing
any induced cube, $K_2\Box K_{1,3},$ or suspended cogwheel. To prove
that $G$ is isometrically embeddable into the Cartesian product of
at most five trees it suffices to show that the incompatibility
graph Inc$({\mathcal S}(G))$  of convex splits of $G$ is
5-colorable. According to a result of De Bruijn and Erd\"os (see
Theorem 1 of the book by Jensen and Toft \cite{JeTo}) an infinite
graph is $k$-colorable if and only if all its finite subgraphs are
$k$-colorable. Consider any finite subgraph $H$ of Inc$({\mathcal
S}(G)).$ For each pair $\sigma_1=\{ A_1,B_1\}$ and $\sigma_2=\{
A_2,B_2\}$ of incompatible splits of $H$ pick a vertex  in each of
the four nonempty intersections $A_1\cap A_2,A_1\cap B_2,B_1\cap
A_2,$ and $B_1\cap B_2.$ Denote by $G'$ the finite median graph
induced by the convex hull of the selected vertices.  Since $G'$ is
a subgraph of $G,$ it does not contain any induced cube, $K_2\Box
K_{1,3},$ or suspended cogwheel. According to Lemma \ref{amalgam},
$G'$ is a finite squaregraph. Theorem~\ref{theorem:five_trees} implies that Inc$({\mathcal
S}(G'))$ is 5-colorable. Since $H$ is an induced subgraph of
Inc$({\mathcal S}(G')),$ we deduce that $H$ is 5-colorable as well.
The corresponding assertions for 3- or 2-colorabilty follow similarly
from 	Theorem~\ref{theorem:five_trees} together with Lemmas \ref{2x2} and \ref{2trees}.
This establishes Theorem~\ref{theorem:five_trees_infinite}.

\section{\'{E}pilogue}
\label{sec:conclusion}

Two-connected squaregraphs constitute an interesting class of planar
graphs with a rich structural theory. In particular, the
coordinatization with no more than five trees makes a number of algorithmic
problems particularly tractable for them.
Another equally attractive feature is that any finite 2-connected squaregraph $G$ is fully determined by the metric on its boundary cycle $C$, from which it is obtained as the minimal extension of $C$ to an absolute retract of bipartite graphs, which effectively constitutes a coordinatization by paths with respect to the supremum norm. Now, it is natural to ask
how the injective hull of a 2-connected squaregraph $G$ with
boundary cycle $C$ looks like and how it is determined. Necessarily,
it includes the vertex set of $G$ and must turn every 4-cycle into a
solid square, that is, it encompasses the geometric realization of
$G$ \cite{BaVdV}. It is then not difficult to see that this
geometric realization is in fact injective and thus constitutes the
injective hull (alias tight span) $T(C,d)$ of the boundary cycle $C$
with respect to the distances in $G.$ This also follows directly
from Theorem 1 of \cite{DrHuMo_Bune}. Within this injective hull,
the (at most five) tree factors extend to solid trees (i.e.,
dendrons) and the discrete boundary cycle extends to the solid
boundary circle. The median hull of the boundary circle yields the
whole space $T(C,d).$ In this way, 2-connected squaregraphs arise as
the 1-skeletons of particular Manhattan orbifolds \cite{Epp_orbi}.
The passage from $C$ to $G$ and to the square complex $T(C,d)$ can
be investigated in a more general framework, and this will be the
topic of subsequent papers.

\section*{Acknowledgements}

Work of V. Chepoi was supported in part by the ANR grant BLAN06-1-138894 (projet
OPTICOMB). Work of D. Eppstein was supported in part by NSF grant 0830403 and by the
Office of Naval Research under grant N00014-08-1-1015.

\raggedright
\bibliographystyle{amsplain}
\bibliography{squaregraphs}

\end{document}